\documentclass[11pt]{amsart}
\usepackage{mathrsfs}
\usepackage{amssymb}
\usepackage{amsthm}
\usepackage{titletoc}
\usepackage{amssymb}
\usepackage{amsxtra}
\usepackage{mathrsfs}
\usepackage{hyperref}

\newtheorem{theorem}{Theorem}
\newtheorem{corollary}{Corollary}

\newtheorem{lemma}{Lemma}
\newtheorem{remark}{Remark}
\newtheorem{definition}{Definition}
\newtheorem{question}{Question}

\title[On existence questions for the functional equations]{ On existence questions for the functional equations $f^8+g^8+h^8=1$ and $f^6+g^6+h^6=1$}
\author{Xiao-Min Li$^1$*, Hong-Xun Yi$^2,$   Risto Korhonen$^3$  }
\address{$^{1}$ Department of Mathematics, Ocean University of China, Qingdao, Shandong 266100, P. R. China
\vskip 2mm \hspace{1.5mm} Email:{\sf lixiaomin@ouc.edu.cn(X.-M. Li)).}
\vskip 2mm $^2$ Department of Mathematics, Shandong University, Jinan, Shandong 250199, P. R. China
\vskip 2mm \hspace{1.5mm} Email:{\sf hxyi@sdu.edu.cn(H.-X. Yi).}
\vskip 2mm $^3$ University of Eastern Finland, Department of Physics and Mathematics, P.O. Box 111,
FI-80101 Joensuu, Finland
\vskip 2mm \hspace{1.5mm}Email:{\sf risto.korhonen@uef.fi (R.-J. Korhonen)}	
}
\thanks{{\sf Corresponding author: Xiao-Min Li}}
\thanks{{\sf 2010 Mathematics Subject Classification.} Primary 30D35; Secondary 30D30.}
\thanks{Project supported in part by the NSF of Shandong Province, China (No. ZR2019MA029) and the FRFCU(No. 3016000/842464005)}
\thanks{{\sf Keywords.} Nevanlinna$'$s theory; Meromorphic functions; Fermat-type functional equations.}
\begin{document}
\vskip 2mm
\par
\maketitle
\begin{abstract}  In 1985, Hayman \cite{Hayman1985} proved that there do not exist non-constant meromorphic functions $f,$ $g$
and $h$ satisfying the functional equation $f^n+g^n+h^n=1$ for $n\geq 9.$ We prove that there do not exist non-constant meromorphic solutions $f,$ $g,$ $h$ satisfying the functional equation $f^8+g^8+h^8=1.$ In 1971, Toda \cite{Toda1971} proved that there do not exist non-constant entire functions $f,$ $g,$ $h$ satisfying $f^n+g^n+h^n=1$ for $n\geq 7.$ We prove that there do not exist non-constant entire functions $f,$ $g,$ $h$ satisfying the functional equation $f^6+g^6+h^6=1.$ Our results answer questions of G. G. Gundersen.
\end{abstract}
\maketitle
\section{Introduction and main results}
\vskip 2mm
\par
In this paper, we consider the problem whether
there exist transcendental meromorphic functions $f,$ $g,$ $h$ which satisfy the
functional equation
\begin{equation}\label{eq1.1}
f^n + g^n + h^n =1,
\end{equation}
where $n \geq 2$ is an integer. By \textit{meromorphic} functions, we will always mean meromorphic functions in the complex plane $\Bbb{C}.$
\par
\vskip 2mm
In 1985, W.K.Hayman proved the following result:
\par
\vskip 2mm
 \begin{theorem}\rm{}(\cite{Hayman1985})\label{Theorem1.1} \textit{There do not exist non-constant meromorphic functions $f,$ $g,$ $h$ satisfying
 \eqref{eq1.1} for $n \geq 9.$}
\end{theorem}
\par
\vskip 2mm
Later on, Gundersen \cite{Gundersen1998} and  Gundersen  \cite{Gundersen2001} gave examples of transcendental meromorphic functions $f,$ $g,$ $h$ satisfying \eqref{eq1.1} for the cases $n=6$ and $n=5.$ For the cases $n = 2, 3, 4,$ it is known that there exist transcendental entire functions $f,$ $g,$ $h$ satisfying \eqref{eq1.1}; see \cite{Gundersen1998}.
\par
\vskip 2mm
  Theorem \ref{Theorem1.1} can also be proved by using earlier results of Fujimoto \cite[Corollary 6.4]{Fujimoto1974} or Green
  \cite[Lemma on p. 71]{Green1975}. For rational functions, Hayman \cite{Hayman1985} proved the next result:
  \par
\vskip 2mm
 \begin{theorem}\rm{(\cite{Hayman1985})}\label{Theorem1.2} \textit{There do not exist non-constant rational functions $f,$ $g,$ $h$ satisfying
 \eqref{eq1.1} for $n \geq 8.$}
\end{theorem}
\par
\vskip 2mm
In 1971, Toda \cite{Toda1971} proved the following result:
\par
\vskip 2mm
 \begin{theorem}\rm{(\cite{Toda1971})}\label{Theorem1.3} \textit{There do not exist non-constant entire functions $f,$ $g,$ $h$ satisfying
 \eqref{eq1.1} for $n \geq 7.$}
\end{theorem}
\par
\vskip 2mm
Gundersen-Tohge \cite{Gundersen2004} gave examples of transcendental entire functions $f,$ $g,$ $h$ satisfying \eqref{eq1.1} for $n= 5.$ It is known that there exist such examples for the cases $n = 2, 3, 4;$ see \cite{Gundersen2017}. For polynomials, Newman-Slater \cite{Newman1979} proved the next result:
\par
\vskip 2mm
 \begin{theorem}\rm{(\cite{Newman1979})}\label{Theorem1.4} \textit{There do not exist non-constant polynomials $f,$ $g,$ $h$ satisfying
 \eqref{eq1.1} for $n \geq 6.$}
\end{theorem}
\par
 \vskip 2mm
  We mention that Hayman \cite{Hayman1985} showed that Theorems \ref{Theorem1.1}-\ref{Theorem1.4} can all be proved by using Cartan's version of Nevanlinna Theory.
  \par
 \vskip 2mm From the above discussion, one may ask the following questions:
 \par
\vskip 2mm
\begin{question}\rm{ (\cite{Gundersen2003} and \cite[Question 3.1]{Gundersen2017})}\label {question1.5} \, whether or not there exist non-constant meromorphic functions $f,$ $g,$ $h$ satisfying $f^n+g^n+h^n=1$ for the cases $n=7$ and $n=8$?
\end{question}
\par
\vskip 2mm
\begin{question} \rm{(\cite{Gundersen2003} and \cite[Question 3.3]{Gundersen2017})}\label{question1.6} \, whether or not there exist non-constant entire functions $f,$ $g,$ $h$ satisfying $f^6+g^6+h^6=1$?
\end{question}
\vskip 2mm
\par
Throughout the paper, we adopt the standard
notations $m(r,f),$ $N(r,f),$ $T(r,f),$ $S(r,f),$ etc., in the Nevanlinna theory of a meromorphic function $f.$ We assume the readers are familiar with the fundamental results and the typical kind of arguments in Nevanlinna theory; see for example, in \cite{Cherry2001, Goldberg2008,  Hayman1964, Laine1993, Nevanlinna1970, YangYi2003, Yang1993}.  It will be convenient to let $E$ denote any set of positive real numbers of finite linear
measure, not necessarily the same set at each occurrence. Thus for a
non-constant meromorphic function $h,$ we denote by $S(r, h)$ any quantity
satisfying $S(r,h)=o(T(r,h)),$ as $r\rightarrow\infty, r\not\in E,$ where $T(r,h)$ is the Nevanlinna characteristic of $h.$  Following \cite[p.466]{HalburdKorhonen2006}, we define the notion of the small function of a non-constant meromorphic function in the complex plane as follows: let $f$ be a non-constant meromorphic function. and let $\alpha$ be also a meromorphic function in the complex plane such that $T(r, \alpha) = o(T(r,f))$ for all $r\in [0,+\infty)$ possibly outside a set $E\subset [0,+\infty)$ of finite linear measure. Then, the meromorphic function $\alpha$ is called small compared to $f,$ or slowly moving with respect to $f,$  or a small function of $f$ for short.
\par
\vskip 2mm
Regarding Questions \ref{question1.5} and \ref{question1.6},  we recall the following results from Ishizaki \cite{Ishizaki2002}:
 \par
\vskip 2mm
 \begin{theorem}\rm{(\cite[Proposition 5.1]{Ishizaki2002})}. \label{TheoremA}
  \textit{Suppose that there exist transcendental meromorphic functions
$f,$ $g,$ $h$ satisfying the functional equation \eqref{eq1.1} for $n = 8,$ and suppose
that $f^8,$ $g^8$ and $h^8$ are linearly independent. Then there exists a small function
$\alpha$ with respect to $f,$ $g$ and $h$ such that
\begin{equation}\label{eq1.2}
W(f^8, g^8, h^8) = \alpha f^6g^6h^6,
\end{equation}
where and in what follows, $W(f^8, g^8, h^8)$ is the
Wronskian determinant of $f^8,$ $g^8$ and $h^8.$
}
 \end{theorem}
 \begin{theorem} \rm{(\cite[Proposition 6.1]{Ishizaki2002})}. \label{TheoremB}
 \textit{Suppose that there exist transcendental functions
$f,$ $g,$ $h$ satisfying the functional equation \eqref{eq1.1} for $n = 6,$ and suppose
that $f^6,$ $g^6$ and $h^6$ are linearly independent. Then there exists a small function
$\beta$ with respect to $f,$ $g$ and $h$ such that
\begin{equation}\label{eq1.3}
W(f^6, g^6, h^6) = \beta f^4g^4h^4.
\end{equation}
}
\end{theorem}
\par
\vskip 2mm  For the existence of non-constant meromorphic solutions of a much more general three term functional equation, we recall the following result from  Hu-Li-Yang \cite{HuLiYang2003} and Yang-Zhang \cite{YangZhang2008, Yang2010}:
 \begin{theorem}\rm{ (\cite{HuLiYang2003, YangZhang2008,Yang2010})}
 \textit{ If
 \begin{equation}\nonumber
 \frac{1}{n}+ \frac{1}{m} +\frac{1}{k}\leq\frac{1}{3},
  \end{equation}
 then there do not exist any non-constant meromorphic functions $f,$ $g,$
$h$ satisfying the functional equation $f^n+g^m+h^k=1,$ where $n,$ $m$ and $k$ are positive integers such that $\min\{n,m,k\}\geq 2.$
}
\end{theorem}
\par
\vskip 2mm
 We will prove the following result:
 \par
\vskip 2mm
 \begin{theorem} \label {Theorem1.10}\rm{}
 \textit{There do not exist transcendental meromorphic solutions $f,$ $g,$ $h$ satisfying the functional equation $f^8+g^8+h^8=1.$}
 \end{theorem}
\par
\vskip 2mm
  In the same manner as in the proof of Theorem \ref{Theorem1.10} in Section 3, we can get the following result:
\begin{theorem} \label {Theorem1.11}\rm{}
\textit{There do not exist transcendental entire solutions $f,$ $g,$ $h$ satisfying the functional equation $f^6+g^6+h^6=1.$}
\end{theorem}
\par
\vskip 2mm
By combining Theorems \ref{Theorem1.2} and \ref{Theorem1.4} with Theorems \ref{Theorem1.10} and \ref{Theorem1.11}, we obtain the following results:
\par
\vskip 2mm
 \begin{corollary} \label {corollary1.12} \rm{}
\textit{ There do not exist non-constant meromorphic solutions $f,$ $g,$ $h$ satisfying the functional equation $f^8+g^8+h^8=1.$}
 \end{corollary}
 \begin{corollary} \label {corollary1.13}\rm{ }
\textit{There do not exist non-constant entire solutions $f,$ $g,$ $h$ satisfying the functional equation $f^6+g^6+h^6=1.$}
\end{corollary}
\par
\vskip 2mm Corollary \ref{corollary1.12} resolves Question \ref{question1.5} for $n=8$ and Corollary \ref{corollary1.13} resolves Question
\ref{question1.6} completely.

\section{Preliminaries}
\vskip 2mm
\par
 In this section, we will introduce some important lemmas to prove the main results in this paper. We first give the following notations: Let $F$ be a non-constant meromorphic function in the complex plane, and let $z_0\in\Bbb{C}$ be a point. Next we set
$$U(F,z_0)=
\left\{ \aligned
p, & \quad \text{if $z_0$ is a pole of $F$ with multiplicity $p\geq 1,$}\\
0, &  \quad \text{if $z_0$ is not a pole of $F.$}
 \endaligned
\right.
$$
and
$$\overline{U}(F,z_0)=
\left\{ \aligned
1, & \quad \text{if $z_0$ is a pole of $F$ with multiplicity $p\geq 1,$}\\
0, &  \quad \text{if $z_0$ is not a pole of $F.$}
 \endaligned
\right.
$$
Similarly, we define
$$U\left(\frac{1}{F},z_0\right)=
\left\{ \aligned
q, & \quad \text{if $z_0$ is a zero of $F$ with multiplicity $q\geq 1,$}\\
0, &  \quad \text{if $z_0$ is not a zero of $F.$}
 \endaligned
\right.
$$
and
$$\overline{U}\left(\frac{1}{F},z_0\right)=
\left\{ \aligned
1, & \quad \text{if $z_0$ is a zero of $F$ with multiplicity $q\geq 1,$}\\
0, &  \quad \text{if $z_0$ is not a zero of $F.$}
 \endaligned
\right.
$$
Let $F_1$ and $F_2$ be two non-constant meromorphic functions, and let
$$
D=\left|\aligned
F'_1 &\quad  F'_2\\
F''_1&\quad  F''_2
\endaligned
\right|=F'_1F''_2-F'_2F''_1.
$$
\vskip 2mm
\par
We recall the following results from Yi\cite{Yi1994}:
\vskip 2mm
\par
\begin{lemma}\rm{}(\cite{Yi1994})\label{lemma2.1} \textit{The following items hold:}
\vskip 2mm
\par
\rm{(i)} \textit{If $U\left(\frac{1}{F_1}, z_0\right)=q\geq 2$ and $U(F_2, z_0)=0,$ then $U\left(\frac{1}{D}, z_0\right)\geq q-2.$}
\vskip 2mm
\par
\rm{(ii)} \textit{If $U\left(\frac{1}{F_1}, z_0\right)=q_1\geq 2$ and $U\left(\frac{1}{F_2}, z_0\right)=q_2\geq 2,$ then $U\left(\frac{1}{D}, z_0\right)\geq q_1+q_2-3.$}
\vskip 2mm
\par
\rm{(iii)} \textit{If $U(F_1, z_0)=U(F_2, z_0)=p,$ then $U(D, z_0)\leq 2p+2.$}
\vskip 2mm
\par
\rm{(iv)}\textit{ If $U(F_1, z_0)=p_1$ and $U(F_2, z_0)=p_2,$ then $U(D, z_0)\leq  p_1+p_2+3.$}
\vskip 2mm
\par
\rm{(v)} \textit{If $U(F_1, z_0)=p$ and $U(F_2, z_0)=0,$ then $U(D, z_0)\leq  p+2.$}
\vskip 2mm
\par
\rm{(vi)} \textit{If $U(\frac{1}{F_1}, z_0)=p>0$ and $U(F_2, z_0)=q>0,$ then $U(D, z_0)-U\left(\frac{1}{D}, z_0\right)\leq  q-p+3.$}
\end{lemma}
\vskip 2mm
\par
Next we let $f,$ $g,$ $h$ be non-constant meromorphic functions satisfying the following functional equation
\begin{equation} \label{eq2.1}
f^n+g^m+h^k=1,
\end{equation}
where and in what follows, $n,$ $m$ and $k$ are positive integers such that $\min\{n,m,k\}\geq 2.$ Next we also set
\begin{equation} \label{eq2.2}
n_1=\max\{n,m,k\}, \quad n_3=\min\{n,m,k\} ,\quad  n_2=n+m+k-n_1-n_3.
\end{equation}
\vskip 2mm
\par
The following Lemma \ref{lemma2.2}-Lemma \ref{lemma2.12} were  proved originally in the Chinese academic paper Yi-Yang\cite{YiYang2011}. For convenience of the readers, we provide their detailed proofs as follows:
\vskip 2mm
\par
\begin{lemma}\label{lemma2.2}
\rm{}\textit{Suppose that $f,$ $g$ and $h$ are non-constant meromorphic functions satisfying \eqref{eq2.1}. If $n_2\geq 4$ and $n_3\geq 3,$ then
\begin{equation} \nonumber
T(r,f)\neq S(r), \quad T(r,g)\neq S(r),\quad T(r,h)\neq S(r),
\end{equation}
where and in what follows, $n_2$ and $n_3$ are defined as in \eqref{eq2.2}, $S(r)$ is any quantity such that $S(r)=o (T(r)),$ as $r\not\in E$ and $r\rightarrow\infty.$ Here $E\subset (0,+\infty)$ is a subset with finite linear measure, and $T(r)=T(r,f)+T(r,g)+T(r,h).$
}
\end{lemma}
\vskip 2mm
\par
\begin{proof} Suppose that $T(r,f)=S(r).$ Then, by \eqref{eq2.1} we have
\begin{equation} \label{eq2.3}
T(r,h^k)=T(r,g^m)+S(r),\quad T(r,f)=S(r,g),\quad T(r,f)=S(r,h),
\end{equation}
which implies that
\begin{equation} \label{eq2.4}
T(r,h)=\frac{m}{k}T(r,g)+S(r,g).
\end{equation}
Noting that $f$ is a non-constant meromorphic function, by rewriting \eqref{eq2.1} we have
\begin{equation} \label{eq2.5}
ag^m+ah^k=1,
\end{equation}
where $a=1/(1-f^n).$ By \eqref{eq2.3}, \eqref{eq2.4}, \eqref{eq2.5} and the second fundamental theorem we have
\begin{equation} \nonumber
\begin{split}
mT(r,g)&=T(r,ag^m )+S(r,g)\\
&\leq \overline{N}(r,ag^m)+\overline{N}\left(r,\frac{1}{ag^m}\right)+\overline{N}\left(r,\frac{1}{ag^m-1}\right)+S(r,g)\\
&=\overline{N}(r,g)+\overline{N}\left(r,\frac{1}{g}\right)+\overline{N}\left(r,\frac{1}{ah^k}\right)+S(r,g)\\
&\leq 2T(r,g)+\overline{N}\left(r,\frac{1}{h}\right)+S(r,g)\leq 2T(r,g)+T(r,h)+S(r,g)\\
&\leq \left(2+\frac{m}{k}\right)T(r,g)+S(r,g),
\end{split}
\end{equation}
this together with $n_2\geq 4,$ $n_3\geq 3$ and \eqref{eq2.2} implies a contradiction. Therefore, $T(r,f)\neq S(r).$ Similarly we have
$T(r,g)\neq S(r)$ and $T(r,h)\neq S(r).$
\end{proof}
\begin{lemma}\label{lemma2.3}
\textit{Suppose that $f,$ $g$ and $h$ are non-constant meromorphic functions satisfying \eqref{eq2.1}. If $n_1\geq 4$ and $n_3\geq 3,$
where $n_1$ and $n_3$ are defined as in \eqref{eq2.2}, then
$f^n,$ $g^m$ and $h^k$ are linearly independent.
}
\end{lemma}
\begin{proof} Without loss of generality, we suppose that $n_1=n,$ $n_2=m$ and $n_3=k.$ Then, by the assumption of Lemma \ref{lemma2.3} we have
$n\geq 4$ and $n\geq m\geq k\geq 3.$ Suppose that $f^n,$ $g^m$ and $h^k$ are linearly dependent. Then, there exist three constants $a,$ $b$ and $c$ satisfying $|a|+|b|+|c|>0,$ say $c\neq 0,$ such that  $af^n+bg^m+ch^k=0.$ Combining this with \eqref{eq2.1} we have
\begin{equation} \label{eq2.6}
\begin{split}
\left(1-\frac{a}{c}\right)f^n+\left(1-\frac{b}{c}\right)g^m=1.
\end{split}
\end{equation}
Noting that $f$ and $g$ are non-constant meromorphic functions, we have by \eqref{eq2.6} that $a\neq c$ and $b\neq c.$ Therefore, by \eqref{eq2.6}
and the second fundamental theorem we have
\begin{equation} \nonumber
 T(r,f)=\frac{m}{n}T(r,g)+O(1)
\end{equation}
 and
 \begin{equation} \nonumber
\begin{split}
mT(r,g)&=T\left(r, \left(1-\frac{b}{c}\right)g^m\right)+S(r,g)\\
&\leq \overline{N}\left(r,\left(1-\frac{b}{c}\right)g^m\right)+\overline{N}\left(r,\frac{1}{ \left(1-\frac{b}{c}\right)g^m }\right)+\overline{N}\left(r,\frac{1}{ \left(1-\frac{b}{c}\right)g^m-1}\right)\\
&\quad +S(r,g)=\overline{N}(r,g)+\overline{N}\left(r,\frac{1}{g}\right)+\overline{N}\left(r,\frac{1}{\left(1-\frac{a}{c}\right)f^n }\right)+S(r,g)\\
&\leq 2T(r,g)+\overline{N}\left(r,\frac{1}{f}\right)+S(r,g)\leq 2T(r,g)+T(r,f)+S(r,g)\\
&\leq \left(2+\frac{m}{n}\right)T(r,g)+S(r,g),
\end{split}
\end{equation}
 which together with $n\geq 4$ and $n\geq m\geq k\geq 3$  gives a contradiction.
 \end{proof}
The following result is due to Nevanlinna \cite{Nevanlinna1929}:
\begin{lemma}\rm{(\cite{Nevanlinna1929})}\label{lemma2.4}
\textit{Suppose that $f_1,$ $f_2,$ $\cdots,$ $f_n$ are $n$ linearly independent meromorphic functions
such that $\sum\limits_{j=1}^nf_j=1.$ Then, for $1\leq j\leq n,$
 \begin{equation} \nonumber
 T(r,f_j)\leq \sum\limits_{k=1}^nN\left(r,\frac{1}{f_k}\right)+N(r,f_j)+N(r,D)-\sum\limits_{k=1}^nN\left(r,f_k\right)-N\left(r,\frac{1}{D}\right)+S(r),
 \end{equation}
where $D=W(f_1,f_2,\cdots, f_n)$ is the Wronskian determinant of $f_1,$ $f_2,$ $\cdots,$ $f_n$ and $S(r)$ is defined in a similar way to Lemma \ref{lemma2.2}.
}
\end{lemma}
\vskip 2mm
\par
Let $f,$ $g,$ $h$ be three non-constant meromorphic functions satisfying \eqref{eq2.1}. Then, the Wronskian determinant of $f^n,$ $g^m$ and  $h^k$ is such that
\begin{equation} \label {eq2.7}
 D=W(f^n, g^m, h^k)=\left|\aligned (f^n)' &\quad  (g^m)'\\
                                   (f^n)'' &\quad  (g^m)''
                                   \endaligned
                                  \right|= \left|\aligned (g^m)' &\quad  (h^k)'\\
                                  (g^m)'' &\quad  (h^k)''\endaligned\right| = \left|\aligned (h^k)' &\quad  (f^n)'\\
                                  (h^k)'' &\quad  (f^n)''\endaligned\right|.
   \end{equation}
\vskip 2mm
\par
By Lemmas \ref{lemma2.3} and \ref{lemma2.4} we can get the following result:
\vskip 2mm
\par
\begin{lemma}\label{lemma2.5} \rm{} \textit{Let $n_1\geq 4$ and $n_3\geq 3.$ If $f,$ $g,$ $h$ are three non-constant meromorphic functions satisfying \eqref{eq2.1}, then
 \begin{equation} \nonumber
 \begin{split}
 &\quad nT(r,f)+mT(r,g)+kT(r,h)\\
 &\leq 3\left(N\left(r,\frac{1}{f^n}\right)+N\left(r,\frac{1}{g^m}\right)+N\left(r,\frac{1}{h^k}\right) \right)+3N(r,D)-3N\left(r,\frac{1}{D}\right)\\
 &\quad -2(N(r,f^n)+N(r,g^m)+N(r,h^k))+S(r),
  \end{split}
 \end{equation}
where $D= W(f^n, g^m, h^k)$ is in \eqref{eq2.7} and $S(r)$ is defined as in Lemma \ref{lemma2.2}.
}
\end{lemma}
\vskip 2mm
\par
For proving the following lemma, we first introduce the following notations: Let $f,$ $g,$ $h$ be three non-constant meromorphic functions satisfying \eqref{eq2.1}, and let $z_0\in\Bbb{C}$ be a point such that $z_0$ is either a zero  or a pole of any one of the three meromorphic functions $f,$ $g$ and $h.$ Then, the Laurent developments of $f,$ $g$ and $h$ at the point $z_0$ are
\begin{equation} \label {eq2.8}
 f(z)=a_p(z-z_0)^p+a_{p+1}(z-z_0)^{p+1}+\cdots,
  \end{equation}
\begin{equation}\label{eq2.9}
 g(z)=b_q(z-z_0)^q+b_{q+1}(z-z_0)^{q+1}+\cdots,
  \end{equation}
and
\begin{equation}\label {eq2.10}
 h(z)=c_t(z-z_0)^t+c_{t+1}(z-z_0)^{t+1}+\cdots
  \end{equation}
respectively, where $a_p,$ $b_q$ and $c_t$ are nonzero constants, and $p,$ $q,$ $t$ are three integers such that $|p|+|q|+|t|>0.$ Next we also let
\begin{equation}\label {eq2.11}
\begin{split}
A(z_0)&=3\left(U\left(\frac{1}{f^n}, z_0\right)+U\left(\frac{1}{g^m},z_0\right)+U\left(\frac{1}{h^k},z_0\right)\right)+3U(D,z_0)\\
&\quad-3U\left(\frac{1}{D},z_0\right)-2(U(f^n,z_0)+U(g^m,z_0)+U(h^k,z_0)),
 \end{split}
  \end{equation}
where $D= W(f^n, g^m, h^k)$ is in \eqref{eq2.7}.
\vskip 2mm
\par
Let $F$ and $G$ be non-constant meromorphic functions, and let $z_0\in\Bbb{C}$ be a point. Next we define
$$\overline{U}\left(\frac{1}{F},\frac{1}{G}, z_0\right)=
\left\{ \aligned
1, & \quad \text{if $z_0$ is a common zero of $F$ and $G,$}\\
0, &  \quad \text{if $z_0$ is not a common zero of $F$ and $G.$}
 \endaligned
\right.
$$
\vskip 2mm
\par
By using Lemma \ref{lemma2.1}, we will prove the next several results:
\vskip 2mm
\par
\begin{lemma}\label{lemma2.6} \rm{} \textit{Suppose that $\max\{p,q,t\}>0$ and $\min\{p,q,t\}\geq 0.$ Then
\begin{equation}\label{eq2.12}
\begin{split}
A(z_0)&\leq 6\overline{U}\left(\frac{1}{f},z_0\right)+6\overline{U}\left(\frac{1}{g},z_0\right)+6\overline{U}\left(\frac{1}{h},z_0\right)
-3\overline{U}\left(\frac{1}{f},\frac{1}{g}, z_0\right)\\
&\quad-3\overline{U}\left(\frac{1}{g},\frac{1}{h},z_0\right)
-3\overline{U}\left(\frac{1}{h},\frac{1}{f},z_0\right).
 \end{split}
  \end{equation}
}
\end{lemma}
\vskip 2mm
\par
\begin{proof} \, Since $f,$ $g,$ $h$ are non-constant meromorphic functions satisfying \eqref{eq2.1}, we deduce by \eqref{eq2.8}- \eqref{eq2.10} and the assumption of Lemma \ref{lemma2.6} that at least one of $p,q,t$ is equal to $0,$ hence
$\min\{p,q,t\}=0.$ We consider the following two cases:
\vskip 2mm
\par {\bf Case 1.} \, Suppose that there exists only one of of $p,q,t$ is a positive integer, say $p>0, q=t=0.$ Then, by \eqref{eq2.7} and Lemma \ref{lemma2.1}(i) we have \begin{equation}\label{eq2.13}
U\left(\frac{1}{D},z_0\right)\geq np-2.
  \end{equation}
By \eqref{eq2.11} and \eqref{eq2.13} we have
\begin{equation}\nonumber
A(z_0)\leq 3np-3(np-2)\leq 6\overline{U}\left(\frac{1}{f},z_0\right),
  \end{equation}
which implies that the conclusion of Lemma \ref{lemma2.6} is valid.
\vskip 2mm
\par {\bf Case 2.} \, Suppose that exactly two of $p,q,t$ are positive integers, say $p>0, q>0$ and $t=0.$ Then, by \eqref{eq2.7} and Lemma \ref{lemma2.1}(ii) we have
 \begin{equation}\label{eq2.14}
U\left(\frac{1}{D},z_0\right)\geq np+mq-3.
  \end{equation}
By \eqref{eq2.11} and \eqref{eq2.14} we have
\begin{equation}\nonumber
A(z_0)\leq 3(np+mq)-3(np+mq-3)
\leq 6\overline{U}\left(\frac{1}{f},z_0\right)+6\overline{U}\left(\frac{1}{g},z_0\right)-3\overline{U}\left(\frac{1}{f}, \frac{1}{g},z_0\right),
  \end{equation}
which reveals the conclusion of Lemma \ref{lemma2.6}.
\end{proof}
\vskip 2mm
\par
\begin{lemma}\label{lemma2.7} \rm{}\textit{Suppose that $\max\{p,q,t\}>0,$ $\min\{p,q,t\}<0$ and $\min\{m,n,k\}\geq 6.$  Then
\begin{equation}\label{eq2.15}
A(z_0)\leq 3\overline{U}\left(\frac{1}{f},z_0\right)+3\overline{U}\left(\frac{1}{g},z_0\right)+3\overline{U}\left(\frac{1}{h},z_0\right).
 \end{equation}
 }
 \end{lemma}
 \vskip 2mm
\par
\begin{proof} Since $f,$ $g,$ $h$ are non-constant meromorphic functions satisfying \eqref{eq2.1}, we deduce by \eqref{eq2.8}-\eqref{eq2.10} and the assumption of Lemma \ref{lemma2.7} that only one of $p,$ $q,$ $t$ is a positive integer, the other two of  $p,q,t$ are negative integers, say $p>0,$ $q<0,$ $t<0$ and $m|q|=k|t|.$ This together with \eqref{eq2.7} and Lemma \ref{lemma2.1}(vi) gives
\begin{equation}\label{eq2.16}
U(D,z_0)-U\left(\frac{1}{D},z_0\right)\leq m|q|-np+3.
  \end{equation}
  By \eqref{eq2.11}, \eqref{eq2.16} and the assumption $\min\{m,n,k\}\geq 6$ we have
  \begin{equation}\nonumber
A(z_0)\leq 3np+3(m|q|-np+3)-4m|q|=9-m|q|\leq 3\overline{U}\left(\frac{1}{f},z_0\right),
  \end{equation}
which reveals the conclusion of Lemma \ref{lemma2.7}.
\end{proof}
\vskip 2mm
\par
\begin{lemma}\label{lemma2.8}  \rm{}\textit{Suppose that $\max\{p,q,t\}=0,$ $\min\{p,q,t\}<0$ and $\min\{m,n,k\}\geq 6.$  Then
\begin{equation}\label{eq2.17}
A(z_0)\leq 0.
 \end{equation}
 }
 \end{lemma}
\vskip 2mm
\par
\begin{proof} Since $f,$ $g,$ $h$ are non-constant meromorphic functions satisfying \eqref{eq2.1}, we deduce by \eqref{eq2.8}- \eqref{eq2.10} and the assumption of Lemma \ref{lemma2.8} that there exists only one of $p,q,t$ that is equal to zero, the other two of  $p,q,t$ are negative integers, say $p=0,$ $q<0,$ $t<0$ and $m|q|=k|t|.$ This together with \eqref{eq2.7} and  Lemma \ref{lemma2.1}(v) gives
\begin{equation}\label{eq2.18}
U(D,z_0)\leq m|q|+2.
  \end{equation}
  By \eqref{eq2.11}, \eqref{eq2.18} and the assumption $\min\{m,n,k\}\geq 6$ we have
  \begin{equation}\nonumber
A(z_0)\leq 3(m|q|+2)-4m|q|=6-m|q|\leq 0,
  \end{equation}
which reveals the conclusion of Lemma \ref{lemma2.8}.
\end{proof}
\vskip 2mm
\par
\begin{lemma}\label{lemma2.9}  \rm{}
\textit{
Suppose that $\max\{p,q,t\}<0$ and
$$\max\{n|p|,m|q|,k|t|\}=\min\{n|p|,m|q|,k|t|\}.$$  Then
\begin{equation}\label{eq2.19}
A(z_0)\leq 6.
 \end{equation}
 }
 \end{lemma}
\vskip 2mm
\par
\begin{proof} Since $f,$ $g,$ $h$ are non-constant meromorphic functions satisfying \eqref{eq2.1}, we deduce by \eqref{eq2.8}- \eqref{eq2.10} and
the assumption
$$\max\{n|p|,m|q|,k|t|\}=\min\{n|p|,m|q|,k|t|\}$$
that $n|p|=m|q|=k|t|.$  This together with \eqref{eq2.7} and Lemma \ref{lemma2.1}(iii) gives
\begin{equation}\label{eq2.20}
U(D,z_0)\leq 2n|p|+2.
  \end{equation}
 By \eqref{eq2.11} and \eqref{eq2.20} we have
  \begin{equation}\nonumber
A(z_0)\leq 3( 2n|p|+2 )-6n|p|=6,
  \end{equation}
which reveals the conclusion of Lemma \ref{lemma2.9}.
\end{proof}
\begin{lemma}\label{lemma2.10}  \rm{} \textit{Suppose that $\max\{p,q,t\}<0$ and
$$\max\{n|p|,m|q|,k|t|\}>\min\{n|p|,m|q|,k|t|\}.$$ If $k|t|=\min\{n|p|,m|q|,k|t|\},$ then
 \begin{equation}\label{eq2.21}
A(z_0)\leq 9-(m|q|-k|t|).
 \end{equation}
 }
 \end{lemma}
\vskip 2mm
\par
\begin{proof} Since $f,$ $g,$ $h$ are non-constant meromorphic functions satisfying \eqref{eq2.1}, we deduce by \eqref{eq2.8}-\eqref{eq2.10} and
the assumption $\max\{n|p|,m|q|,k|t|\}>\min\{n|p|,m|q|,k|t|\}$ and $k|t|=\min\{n|p|, m|q|, k|t|\}$
that $n|p|=m|q|>k|t|.$ This together with \eqref{eq2.7} and Lemma \ref{lemma2.1}(iv) gives
\begin{equation}\label{eq2.22}
U(D,z_0)\leq m|q|+|k|t|+3.
  \end{equation}
 By \eqref{eq2.11} and \eqref{eq2.22} we have
  \begin{equation}\nonumber
A(z_0)\leq 3(m|q|+k|t|+3)-2(2m|q|+k|t|)=9-(m|q|-k|t|),
  \end{equation}
which reveals the conclusion of Lemma \ref{lemma2.10}.
\end{proof}
\vskip 2mm
\par
Let $f_1,$ $f_2,$ $f_3$ be three non-constant meromorphic functions. Next we denote by $\overline{N}_0(r, f_1, f_2, f_3)$
the reduced counting function of common poles of $f_1,$ $f_2$ and $f_3$ in $|z|<r,$ and denote by  $\overline{N}_0(r, 1/f_1, 1/f_2)$ the reduced
counting function of common zeros of $f_1$ and $f_2$ in $|z|<r.$
\vskip 2mm
\par
The following result plays an important role in proving Theorem \ref{Theorem1.10} of this paper:
\begin{lemma}\label{lemma2.11}\textit{Suppose that there exist non-constant meromorphic functions
$f,$ $g$ and $h$ satisfying the functional equation \eqref{eq1.1} for $n = 8.$ Then
 \begin{equation}\label{eq2.23}
 \begin{split}
 \frac{1}{3}T(r)&=T(r,f)+S(r)=T(r,g)+S(r)=T(r,h)+S(r)\\
&=\overline{N}\left(r,\frac{1}{f}\right)+S(r)=\overline{N}\left(r,\frac{1}{g}\right)+S(r)=\overline{N}\left(r,\frac{1}{h}\right)+S(r)\\
&=\overline{N}(r,f)+S(r)=\overline{N}(r,g)+S(r)=\overline{N}(r,h)+S(r)\\
 &=\overline{N}_0(r,f, g,h)+S(r)\\
 \end{split}
 \end{equation}
 and
\begin{equation}\label{eq2.24}
 \overline{N}_0\left(r,\frac{1}{f}, \frac{1}{g}\right)+\overline{N}_0\left(r,\frac{1}{g}, \frac{1}{h}\right)+\overline{N}_0\left(r,\frac{1}{h}, \frac{1}{f}\right)=S(r),
\end{equation}
where $T(r)$ and $S(r)$ are defined as in Lemma \ref{lemma2.2}.
}
  \end{lemma}
\vskip 2mm
\par
 \begin{proof}
 We first set
 $$\overline{U}\left(f,g, h, z_0\right)=
\left\{ \aligned
1, & \quad \text{if $z_0$ is a common pole of $f,$ $g$ and $h,$}\\
0, &  \quad \text{if $z_0$ is not a common pole of $f,$ $g$ and $h.$ }
 \endaligned
\right.
$$
Since $f,$ $g,$ $h$ are non-constant meromorphic functions satisfying \eqref{eq1.1} for $n=8,$  by Lemma \ref{lemma2.5} we have
   \begin{equation} \label{eq2.25}
 \begin{split}
 &\quad 8T(r,f)+8T(r,g)+8T(r,h)\\
 &\leq 3\left(N\left(r,\frac{1}{f^8}\right)+N\left(r,\frac{1}{g^8}\right)+N\left(r,\frac{1}{h^8}\right) \right)+3N(r,D)-3N\left(r,\frac{1}{D}\right)\\
 &\quad -2(N(r,f^8)+N(r,g^8)+N(r,h^8))+S(r),
  \end{split}
 \end{equation}
where $D= W(f^8, g^8, h^8)$ is the Wronskian determinant of $f^8,$ $g^8$ and $h^8,$ and $S(r)$ is defined as in Lemma \ref{lemma2.2}. Moreover, \eqref{eq2.11} can be rewritten as
  \begin{equation}\label {eq2.26}
\begin{split}
A(z_0)&=3\left(U\left(\frac{1}{f^8}, z_0\right)+U\left(\frac{1}{g^8},z_0\right)+U\left(\frac{1}{h^8},z_0\right)\right)+3U(D,z_0)\\
&\quad-3U\left(\frac{1}{D},z_0\right)-2(U(f^8,z_0)+U(g^8,z_0)+U(h^8,z_0)).
 \end{split}
  \end{equation}
 By \eqref{eq2.26} and Lemmas \ref{lemma2.6}-\ref{lemma2.10} we deduce
  \begin{equation}\label {eq2.27}
  \begin{split}
A(z_0)&\leq 6\left(\overline{U}\left(\frac{1}{f}, z_0\right)+\overline{U}\left(\frac{1}{g},z_0\right)+\overline{U}\left(\frac{1}{h},z_0\right)\right)\\
    &\quad -3\left(\overline{U}\left(\frac{1}{f}, \frac{1}{g}, z_0\right)+\overline{U}\left(\frac{1}{g},\frac{1}{h}, z_0\right)+\overline{U}\left(\frac{1}{h},\frac{1}{f}, z_0\right)\right)+6\overline{U}\left(f,g, h, z_0\right).\\
    \end{split}
  \end{equation}
 From \eqref{eq2.25}-\eqref{eq2.27} we have
 \begin{equation}\nonumber
     \begin{split}
     &\quad 8T(r,f)+8T(r,g)+8T(r,h)\\
   &\leq 3\left(N\left(r,\frac{1}{f^8}\right)+N\left(r,\frac{1}{g^8}\right)+N\left(r,\frac{1}{h^8}\right) \right)+3N(r,D)-3N\left(r,\frac{1}{D}\right)\\
 &\quad -2(N(r,f^8)+N(r,g^8)+N(r,h^8))+S(r)\\
 &\leq 6\left(\overline{N}\left(r,\frac{1}{f}\right)+\overline{N}\left(r,\frac{1}{g}\right)+\overline{N}\left(r,\frac{1}{h}\right)\right)\\
    &\quad -3\left(\overline{N}_0\left(r, \frac{1}{f}, \frac{1}{g}\right)+\overline{N}_0\left(r, \frac{1}{g},\frac{1}{h}\right)+\overline{N}_0\left(r, \frac{1}{h},\frac{1}{f}\right)\right)+ 6\overline{N}_0\left(r, f, g, h\right)+S(r)\\
    &\leq 6\left(N\left(r,\frac{1}{f}\right)+N\left(r,\frac{1}{g}\right)+N\left(r,\frac{1}{h}\right)\right)\\
    &\quad -3\left(\overline{N}_0\left(r, \frac{1}{f}, \frac{1}{g}\right)
    +\overline{N}_0\left(r, \frac{1}{g},\frac{1}{h}\right)+\overline{N}_0\left(r, \frac{1}{h},\frac{1}{f}\right)\right)\\
 &\quad+2(\overline{N}\left(r,f\right)+\overline{N}\left(r, g\right)+\overline{N}\left(r, h\right))+S(r)\\
 &\leq 6\left(T\left(r,f\right)+T\left(r,g\right)+T\left(r,h\right)\right)\\
 &\quad -3\left(\overline{N}_0\left(r, \frac{1}{f}, \frac{1}{g}\right)+\overline{N}_0\left(r, \frac{1}{g},\frac{1}{h}\right)+\overline{N}_0\left(r, \frac{1}{h},\frac{1}{f}\right)\right)\\
&\quad+2(N\left(r,f\right)+N\left(r, g\right)+N\left(r, h\right))+S(r)\\
&\leq 8\left(T\left(r,f\right)+T\left(r,g\right)+T\left(r,h\right)\right)\\
&\quad -3\left(\overline{N}_0\left(r, \frac{1}{f}, \frac{1}{g}\right)+\overline{N}_0\left(r, \frac{1}{g},\frac{1}{h}\right)+\overline{N}_0\left(r, \frac{1}{h},\frac{1}{f}\right)\right)+S(r),
   \end{split}
  \end{equation}
   which implies that
  \begin{equation}\label{eq2.28}
   \overline{N}\left(r,f\right)= N\left(r,f\right)+S(r)=\overline{N}_0(r,f,g,h)+S(r)=T(r,f)+S(r),
    \end{equation}
   \begin{equation}\label{eq2.29}
   \overline{N}\left(r,g\right)= N\left(r,g\right)+S(r)=\overline{N}_0(r,f,g,h)+S(r)=T(r,g)+S(r),
    \end{equation}
  \begin{equation}\label{eq2.30}
   \overline{N}\left(r,h\right)= N\left(r,h\right)+S(r)=\overline{N}_0(r,f,g,h)+S(r)=T(r,h)+S(r),
    \end{equation}
 \begin{equation}\label{eq2.31}
   \overline{N}\left(r,\frac{1}{f}\right)= N\left(r,\frac{1}{f}\right)+S(r)=T(r,f)+S(r),
    \end{equation}
   \begin{equation}\label{eq2.32}
   \overline{N}\left(r,\frac{1}{g}\right)= N\left(r,\frac{1}{g}\right)+S(r)=T(r,g)+S(r),
    \end{equation}
 \begin{equation}\label{eq2.33}
   \overline{N}\left(r,\frac{1}{h}\right)= N\left(r,\frac{1}{h}\right)+S(r)=T(r,h)+S(r)
    \end{equation}
and
  \begin{equation}\label{eq2.34}
   \overline{N}_0\left(r, \frac{1}{f}, \frac{1}{g}\right)+\overline{N}_0\left(r, \frac{1}{g},\frac{1}{h}\right)+\overline{N}_0\left(r, \frac{1}{h},\frac{1}{f}\right)=S(r).
    \end{equation}
   By \eqref{eq2.28}-\eqref{eq2.34} we get the conclusion of Lemma \ref{lemma2.11}.
  \end{proof}
\vskip 2mm
\par
In the same manner as in the proof of Lemma \ref{lemma2.11} we can get the following two results that is used to prove Theorem \ref{Theorem1.11}:
 \begin{lemma} \label{lemma2.12}
 \textit{Suppose that there exist non-constant entire functions
$f,$ $g$ and $h$ satisfying the functional equation \eqref{eq1.1} for $n = 6.$ Then
 \begin{equation}\label{eq2.35}
  \begin{split}
\frac{1}{3}T(r)&=T(r,f)+S(r)=T(r,g)+S(r)=T(r,h)+S(r)=\overline{N}\left(r,\frac{1}{f}\right)+S(r)\\
&=\overline{N}\left(r,\frac{1}{g}\right)+S(r)=\overline{N}\left(r,\frac{1}{h}\right)+S(r)
 \end{split}
 \end{equation}
and
 \begin{equation}\label{eq2.36}
\overline{N}_0\left(r,\frac{1}{f}, \frac{1}{g}\right)+\overline{N}_0\left(r,\frac{1}{g}, \frac{1}{h}\right)+\overline{N}_0\left(r,\frac{1}{h}, \frac{1}{f}\right)=S(r),
\end{equation}
where $T(r)$ and $S(r)$ are defined as in Lemma \ref{lemma2.2}.
}
\end{lemma}
\vskip 2mm
\par \begin{remark}\rm{}\label{remar1.1}Ishizaki \cite[p.82]{Ishizaki2002} proved \eqref{eq2.23} in Lemma \ref{lemma2.11}. Ishizaki \cite[p. 84]{Ishizaki2002} also  proved a part of Lemma \ref{lemma2.12}.
\end{remark}
\vskip 2mm
\par Next we introduce some notions of normal families of meromorphic functions in a domain of the complex plane. Following Montel \cite{Montel1927}, a class $\mathcal{F}$ of meromorphic functions in a domain $D\subseteq\Bbb{C}$ is called normal in $D,$ if given any sequence $\{f_n(z)\}$ of meromorphic functions in  $\mathcal{F},$ we can find a subsequence $\{f_{n_p}(z)\}\subseteq\{f_n(z)\}$ which converges everywhere in $D$ and uniformly on compact subsets of $D$ with respect to the chordal metric on Riemann sphere. We then say that $\{f_{n_p}(z)\}$  converges locally uniformly in $D.$
\vskip 2mm
\par An equivalent statement is that for every $z_0$ in $D$ there exists a neighbourhood $|z-z_0|<\delta$ in which $\{f_{n_p}(z)\}$ or  $\{f^{-1}_{n_p}(z)\}$ converges uniformly as $p\rightarrow\infty.$ Suppose in fact that the second condition is satisfied. If $w_1,$ $w_2$ are two points in the $w$-plane, their distance in the chordal metric of Riemann sphere is
 \begin{equation}\nonumber
k(w_1, w_2)=\frac{|w_1-w_2|}{\sqrt{(1+|w_1|^2)(1+|w_2|^2)}}\leq |w_1-w_2|
\end{equation}
Also
 \begin{equation}\nonumber
k(w_1, w_2)=\frac{|\frac{1}{w_1}-\frac{1}{w_2}|}{\sqrt{(1+|\frac{1}{w_1}|^2)(1+|\frac{1}{w_2}|^2)}}\leq \left|\frac{1}{w_1}-\frac{1}{w_2}\right|.
\end{equation}
Thus, if either $f_n(z)\rightarrow f(z)$ or $f_n(z)^{-1}\rightarrow f(z)^{-1}$ uniformly in a set $E\subseteq\Bbb{C}\cup\{\infty\},$ then
$k(f_n(z), f(z))\rightarrow 0,$ uniformly in $E.$ Thus, if one of these two conditions holds uniformly in some neighbourhood of every point of $D,$
then $k(f_n(z), f(z))\rightarrow 0$ uniformly in some neighbourhood of every point of $D,$  and hence by the Heine-Borel theorem uniformly on every compact subset of $D;$ see, for example, \cite[pp.157-160]{Hayman1964}. In addition, we need the following notion of the spherical derivative of a meromorphic function in the complex plane (cf.\cite{Hayman1964,Yang1993}: let $f$ be a non-constant meromorphic function. The spherical derivative of
$f$ at $z\in \Bbb{C}$ is given as $f^{\#}(z)=\frac{|f'(z)|}{1+|f(z)|^2}.$ We recall the following classical normality criterion of a class of meromorphic functions due to Marty\cite{Marty1931}:
\vskip 2pt
\par
\begin{lemma}\rm{}(\cite{Marty1931})\label{lemma2.13} \textit{A class $\mathcal{F}$ of functions $f$ meromorphic in a domain $D$ of the complex plane is normal in $D$ if and only if $|f'(z)|/(1+|f(z)|^2)$ is uniformly bounded on any compact subset of $D$ for $f\in\mathcal{F}.$}
\end{lemma}
\vskip 2mm
\par The following result is due to Pang \cite{Pang1989}:
\vskip 2mm
\par
\begin{lemma}\rm{}(\cite[Lemma 1]{Pang1989})\label{lemma2.14}
Let $f$ be a meromorphic function in the domain $D=\{z\in\Bbb{C}: |z|<1\},$ and let $\alpha$ be a real number satisfying $0\leq \alpha<1.$ If there exists a point $z^{*}\in\{z\in\Bbb{C}: |z|<r\},$ where $r$ is a fixed positive number satisfying $0<r<1,$ such that
\begin{equation} \nonumber
\frac{\left(1-\left|\frac{z^{*}}{r}\right|^2\right)^{\alpha+1}|f'(z^{*})|}{\left(1-\left|\frac{z^{*}}{r}\right|^2\right)^{2\alpha}+|f(z^{*})|^2}>1,
\end{equation}
then there exist a point $z_0\in\{z\in\Bbb{C}: |z|<r\}$ and a real number $t$ satisfying $0<t<1,$ such that
\begin{equation} \nonumber
\sup\limits_{|z|<r}\frac{\left(1-\left|\frac{z}{r}\right|^2\right)^{\alpha+1}t^{k+1}|f'(z)|}{\left(1-\left|\frac{z}{r}\right|^2\right)^{2\alpha}t^{2\alpha}
+|f(z)|^2}
=\frac{\left(1-\left|\frac{z_0}{r}\right|^2\right)^{\alpha+1}t^{\alpha+1}|f'(z_0)|}{\left(1-\left|\frac{z_0}{r}\right|^2\right)^{2\alpha}t^{2\alpha}
+|f(z_0)|^2}=1.
\end{equation}
\end{lemma}
\vskip 2mm
\par The following result is the well-known Pang-Zalcman's Lemma:
\vskip 2mm
\par
\begin{lemma}\rm{}(Pang-Zalcman's Lemma, \cite{Gu1991,Pang1988,Zalcman1975})\label{lemma2.15}.
 \textit{Let $\mathcal{F}$ be a family of meromorphic functions in the unit disc $D_1(0)=:\{z\in\Bbb{C}:|z|<1\}$ and $\alpha$ be a real number
satisfying $-1<\alpha<1.$ Then, if $\mathcal{F}$ is not normal at a point
$z_0\in D_1(0),$ there exist, for each $-1 <\alpha < 1:$}
\vskip 2mm
\par (i)  \textit{points $z_n\in D_1(0),$ $z_n\rightarrow z_0,$}
\vskip 2mm
\par (ii) \textit{ positive numbers $\rho_n,$ $\rho_n\rightarrow 0^{+}$}
and
\vskip 2mm
\par (iii)  \textit{ functions $f_n\in \mathcal{F}$ such that $\frac{f_n(z_n+\rho_n\zeta)}{\rho_n^{\alpha}}\rightarrow
g(\zeta)$ spherically uniformly on compact subset of $\Bbb{C},$
where $g$ is a non-constant meromorphic function. The function $g$
may be taken to satisfy the normalization $g^{\#}(\zeta)\leq
g^{\#}(0)=1.$}
\end{lemma}
\vskip 2mm
\par
\begin{remark}\rm{} \label{remark2.2} {\bf(I)}  Zalcman \cite[the main lemma]{Zalcman1975} proved Lemma \ref{lemma2.15} for $\alpha=0,$  while Pang \cite[Lemma 3]{Pang1988} and Pang \cite[Lemma 2]{Pang1989} proved Lemma \ref{lemma2.15} for $-1<\alpha<0$ and $0<\alpha<1$ respectively. {\bf (II)} For the convenience of the readers, we provide the proof of Lemma \ref{lemma2.15} from Pang \cite[Lemma 2]{Pang1989} and Gu\cite[Theorem 3.26]{Gu1991} for the case of $0\leq\alpha<1$ as follows: without loss of generality, we suppose that $z_0=0.$ Then, by Lemma \ref{lemma2.13} we see that there exists an infinite sequence of the points $\{z^{*}_n\}$ such that
 $z^{*}_n\in \{z:|z|\leq \frac{1}{2n}\},$ and there exists an infinite sequence of meromorphic functions $\{f_n\}\subset \mathcal{F}$ with $n\in\Bbb{Z}^{+}$ such that
 \begin{equation}\label{eq2.37}
 \frac{|f'_n(z^{*}_n)|}{1+f_n(z^{*}_n)|^2}>2^{1+\alpha}n^{2+\alpha}\quad\text{and}\quad \lim\limits_{n\rightarrow\infty}\frac{|f'_n(z^{*}_n)|}{1+f_n(z^{*}_n)|^2}=\infty.
 \end{equation}
 Next we set
  \begin{equation}\label{eq2.38}
 r^{*}_n=\frac{1}{2n} \quad \text{and}\quad r_n=\frac{1}{n} \quad \text{ with} \quad n\in\Bbb{Z}^{+} \quad \text {and} \quad n\geq 2.
 \end{equation}
 Then, it follows by \eqref{eq2.38} that
 \begin{equation}\label{eq2.39}
  |z^{*}_n|\leq r^{*}_n<r_n<1 \quad  \text{for}\quad n\in\Bbb{Z}^{+}\setminus\{1\}
   \end{equation}
  and
  \begin{equation} \label{eq2.40}
\frac{\left(1-\left|\frac{z^{*}_n}{r_n}\right|^2\right)^{\alpha+1}|f'_n(z^{*}_n)|}{\left(1-\left|\frac{z^{*}_n}{r_n}\right|^2\right)^{2\alpha}+|f_n(z^{*}_n)|^2}
\geq \left(1-\left|\frac{z^{*}_n}{r_n}\right|^2\right)^{\alpha+1}\cdot\frac{|f'_n(z^{*}_n)|}{1+|f_n(z^{*}_n)|^2} ,
\end{equation}
From \eqref{eq2.37}-\eqref{eq2.40} we derive
\begin{equation} \label{eq2.41}
\lim\limits_{n\rightarrow\infty}\frac{\left(1-\left|\frac{z^{*}_n}{r_n}\right|^2\right)^{\alpha+1}|f'_n(z^{*}_n)|}{\left(1-\left|\frac{z^{*}_n}{r_n}\right|^2\right)^{2\alpha}+|f_n(z^{*}_n)|^2}
=\infty.
\end{equation}
 From \eqref{eq2.41} we might as well as assume that
 \begin{equation} \label{eq2.42}
\frac{\left(1-\left|\frac{z^{*}_n}{r_n}\right|^2\right)^{\alpha+1}|f'_n(z^{*}_n)|}{\left(1-\left|\frac{z^{*}_n}{r_n}\right|^2\right)^{2\alpha}+|f_n(z^{*}_n)|^2}
>1.
\end{equation}
 According to \eqref{eq2.42} and Lemma \ref{lemma2.14}, there exist point $z_n\in\{z\in\Bbb{C}: |z|<r_n\}$ and a positive number $t_n$ satisfying $0<t_n<1,$ such that
\begin{equation}\label{eq2.43}
\sup\limits_{|z|<r_n}
\frac{\left(1-\left|\frac{z}{r_n}\right|^2\right)^{\alpha+1}t_n^{k+1}|f'_n(z)|}{\left(1-\left|\frac{z}{r_n}\right|^2\right)^{2\alpha}t_n^{2\alpha}
+|f_n(z)|^2}
=\frac{\left(1-\left|\frac{z_n}{r_n}\right|^2\right)^{\alpha+1}t_n^{\alpha+1}|f'_n(z_n)|}{\left(1-\left|\frac{z_n}{r_n}\right|^2\right)^{2\alpha}t_n^{2\alpha}
+|f_n(z_n)|^2}=1.
\end{equation}
By \eqref{eq2.38}, \eqref{eq2.39}, \eqref{eq2.43} and the first inequality of \eqref{eq2.37} we derive
\begin{equation}\label{eq2.44}
\begin{split}
1&=\frac{\left(1-\left|\frac{z_n}{r_n}\right|^2\right)^{\alpha+1}t_n^{\alpha+1}|f'_n(z_n)|}{\left(1-\left|\frac{z_n}{r_n}\right|^2\right)^{2\alpha}t_n^{2\alpha}
+|f_n(z_n)|^2}\geq \frac{\left(1-\left|\frac{z^{*}_n}{r_n}\right|^2\right)^{\alpha+1}t_n^{\alpha+1}|f'_n(z^{*}_n)|}{\left(1-\left|\frac{z^{*}_n}{r_n}\right|^2\right)^{2\alpha}
t_n^{2\alpha}+|f_n(z^{*}_n)|^2} \\
&\geq t_n^{\alpha+1}\left(1-\left|\frac{z^{*}_n}{r_n}\right|^2\right)^{\alpha+1} \frac{|f'_n(z^{*}_n)|}{1+|f_n(z^{*}_n)|^2}
> 2^{\alpha+1}\left(1-\left|\frac{z^{*}_n}{r_n}\right|^2\right)^{\alpha+1}t_n^{\alpha+1}n^{\alpha+2}\\
&=2^{\alpha+1}n\left(1-\left|\frac{z^{*}_n}{r_n}\right|^2\right)^{\alpha+1}(nt_n)^{\alpha+1}= 2^{\alpha+1}n^{\alpha+2}\left(\left(1-\left|\frac{z^{*}_n}{r_n}\right|^2\right)t_n\right)^{\alpha+1}
\end{split}
\end{equation}
and
\begin{equation}\label{eq2.45}
0\leq \left|\frac{z^{*}_n}{r_n}\right|\leq \frac{1}{2} \quad \text{for each}\quad n\in\Bbb{Z}^{+}.
\end{equation}
By \eqref{eq2.44}, \eqref{eq2.45} and the assumption $0\leq\alpha<1$ we derive
\begin{equation}\label{eq2.46}
\lim\limits_{n\rightarrow\infty}\left(1-\left|\frac{z^{*}_n}{r_n}\right|^2\right)t_n=0\quad\text{and}\quad \lim\limits_{n\rightarrow\infty}nt_n=0.
\end{equation}
By the equality in \eqref{eq2.44} we derive
\begin{equation}\label{eq2.47}
1=\frac{\left(1-\left|\frac{z_n}{r_n}\right|^2\right)^{\alpha+1}t_n^{\alpha+1}|f'_n(z_n)|}{\left(1-\left|\frac{z_n}{r_n}\right|^2\right)^{2\alpha}t_n^{2\alpha}
+|f_n(z_n)|^2}\leq t_n^{1-\alpha}\left(1-\left|\frac{z_n}{r_n}\right|^2\right)^{1-\alpha} \frac{|f'_n(z_n)|}{1+|f_n(z_n)|^2}
\end{equation}
By the right equality of \eqref{eq2.46} we have
\begin{equation}\label{eq2.48}
\lim\limits_{n\rightarrow\infty} t_n=0^{+}.
 \end{equation}
  By \eqref{eq2.47}, \eqref{eq2.48}, the assumption $0\leq \alpha<1$ and the known results $0<t_n<1$ and $|z_n|<r_n$ we derive
\begin{equation}\nonumber 
\lim\limits_{n\rightarrow\infty}\frac{|f'_n(z_n)|}{1+|f_n(z_n)|^2}=\infty.
\end{equation}
Next we set
 \begin{equation}\label{eq2.49}
\rho_n=:\left(1-\left|\frac{z_n}{r_n}\right|^2\right)t_n.
\end{equation}
Then, it follows by \eqref{eq2.48}, \eqref{eq2.49} and the known result $|z_n|<r_n$ that
\begin{equation}\label{eq2.50}
\lim\limits_{n\rightarrow\infty}\frac{\rho_n}{r_n-|z_n|}=0.
\end{equation}
By \eqref{eq2.50} and the fact $|z_n|<r_n$ for $n\in \Bbb{Z}^{+}\setminus\{1\},$ we derive
\begin{equation}\label{eq2.51}
|z_n+\rho_n\zeta|<r_n\quad \text{for each} \quad \zeta\in D_{R_n}(0),
\end{equation}
where and in what follows, $D_{R_n}(0)=\{\zeta\in\Bbb{C}:|\zeta|<R_n\}$ and $R_n=:\frac{r_n-|z_n|}{\rho_n}$ for $n\in \Bbb{Z}^{+}\setminus\{1\},$ such that
\begin{equation}\label{eq2.52}
\lim\limits_{n\rightarrow\infty}R_n=\lim\limits_{n\rightarrow\infty}\frac{r_n-|z_n|}{\rho_n}=\infty.
\end{equation}
By \eqref{eq2.38} and \eqref{eq2.51} we see that the functions
\begin{equation}\label{eq2.53}
g_n(\zeta)=\rho^{-\alpha}_nf_n(z_n+\rho_n\zeta)\quad \text{with}\quad n\in \Bbb{Z}^{+}\setminus \{1\}
\end{equation}
are well defined in $|\zeta|<R_n.$ For any given positive constant $R,$ there exists some large positive integer $N_{R}$ that depends only upon $R$ such that $R<R_n$ when $n>N_{R}.$ Then, for each $\zeta\in \overline{D}_R(0),$ we have $|\zeta|<R_n$ when $n>N_{R},$ where and in what follows, $\overline{D}_R(0)=:\{\zeta\in\Bbb{C}: |\zeta|\leq R\}.$ Combining this with $R_n=\frac{r_n-|z_n|}{\rho_n},$ we derive $|z_n+\rho_n\zeta|<r_n$ for each $\zeta\in \overline{D}_R(0)$ when $n>N_{R}.$ This together with \eqref{eq2.49} and \eqref{eq2.53} gives
\begin{equation}\label{eq2.54}
\frac{|g'_n(\zeta)|}{1+|g_n(\zeta)|^2}=\frac{\left(1-\left|\frac{z_n}{r_n}\right|^2\right)^{\alpha+1}t_n^{\alpha+1}|f'_n(z_n+\rho_n\zeta)|}
{\left(1-\left|\frac{z_n}{r_n}\right|^2\right)^{2\alpha}t_n^{2\alpha}
+|f_n(z_n+\rho_n\zeta)|^2}
\end{equation}
for each $\zeta\in \overline{D}_R(0)$ when $n>N_{R}.$ By \eqref{eq2.43} and \eqref{eq2.54} we deduce
\begin{equation}\label{eq2.55}
\frac{|g'_n(0)|}{1+|g_n(0)|^2}=\frac{\left(1-\left|\frac{z_n}{r_n}\right|^2\right)^{\alpha+1}t_n^{\alpha+1}|f'_n(z_n)|}{\left(1-\left|\frac{z_n}{r_n}\right|^2\right)^{2\alpha}t_n^{2\alpha}
+|f_n(z_n)|^2}=1 \quad\text{with}\quad n\in\Bbb{Z}^{+}\setminus\{1\}.
\end{equation}
On the other hand, for each $\zeta\in \overline{D}_R(0)$ we have
\begin{equation}\label{eq2.56}
|z_n|^2-2R\rho_n-R^2\rho^2_n\leq |z_n+\rho_n\zeta|^2\leq |z_n|^2+2R\rho_n+R^2\rho^2_n
\end{equation}
By \eqref{eq2.50} and \eqref{eq2.56} we derive
\begin{equation}\label{eq2.57}
\lim\limits_{n\rightarrow\infty}\frac{r_n-|z_n|^2}{r_n-|z_n+\rho_n\zeta|^2}=1,
\end{equation}
uniformly on the compact set $\overline{D}_R(0)\subset \Bbb{C}.$  Therefore, it follows by \eqref{eq2.57} that for any large positive integer $n,$ say $n>N_R,$ there exists a positive quantity $\varepsilon_n$ that depends only upon $n,$ such that
\begin{equation}\label{eq2.58}
1-\varepsilon_n<\frac{r_n-|z_n|^2}{r_n-|z_n+\rho_n\zeta|^2}<1+\varepsilon_n\quad\text{and}\quad \lim\limits_{n\rightarrow\infty}\varepsilon_n=0.
\end{equation}
By \eqref{eq2.43}, \eqref{eq2.51}, \eqref{eq2.54} and \eqref{eq2.58} we derive
\begin{equation}\label{eq2.59}
\begin{split}
&\quad \frac{|g'_n(\zeta)|}{1+|g_n(\zeta)|^2}=\frac{\left(1-\left|\frac{z_n}{r_n}\right|^2\right)^{\alpha+1}t_n^{\alpha+1}|f'_n(z_n+\rho_n\zeta)|}
{\left(1-\left|\frac{z_n}{r_n}\right|^2\right)^{2\alpha}t_n^{2\alpha}
+|f_n(z_n+\rho_n\zeta)|^2}\\
&\leq \frac{(1+\varepsilon_n)^{1+\alpha}}{(1-\varepsilon_n)^{2\alpha}}\cdot \frac{\left(1-\left|\frac{z_n+\rho_n\zeta}{r_n}\right|^2\right)^{\alpha+1}t_n^{\alpha+1}|f'_n(z_n+\rho_n\zeta)|}
{\left(1-\left|\frac{z_n+\rho_n\zeta}{r_n}\right|^2\right)^{2\alpha}t_n^{2\alpha}
+|f_n(z_n+\rho_n\zeta)|^2}\\
&\leq \frac{(1+\varepsilon_n)^{1+\alpha}}{(1-\varepsilon_n)^{2\alpha}}\sup\limits_{|z|<r_n}\frac{\left(1-\left|\frac{z}{r_n}\right|^2\right)^{\alpha+1}t_n^{k+1}|f'_n(z)|}{\left(1-\left|\frac{z}{r_n}\right|^2\right)^{2\alpha}t_n^{2\alpha}
+|f_n(z)|^2}\\
\end{split}
\end{equation}
\begin{equation}\nonumber
\begin{split}
&=\frac{(1+\varepsilon_n)^{1+\alpha}}{(1-\varepsilon_n)^{2\alpha}}\cdot\frac{\left(1-\left|\frac{z_n}{r_n}\right|^2\right)^{\alpha+1}t_n^{\alpha+1}|f'_n(z_n)|}{\left(1-\left|\frac{z_n}{r_n}\right|^2\right)^{2\alpha}t_n^{2\alpha}
+|f_n(z_n)|^2}=\frac{(1+\varepsilon_n)^{1+\alpha}}{(1-\varepsilon_n)^{2\alpha}}<2\\
\end{split}
\end{equation}
for each $\zeta\in \overline{D}_R(0)$ and any large positive integer $n$ satisfying $n>N_R.$ Since $R$ is any given positive number, we have by \eqref{eq2.59} and Lemma \ref{lemma2.8} that $\{g_n\}$ is normal in $\Bbb{C}.$ Without loss of generality, we suppose that  $\{g_n\}$ itself converges to a meromorphic function $g$ spherically uniformly on any compact subset of $\Bbb{C}.$ By \eqref{eq2.55} we have
\begin{equation}\label{eq2.60}
\frac{|g'(0)|}{1+|g(0)|^2}=\lim\limits_{n\rightarrow\infty}\frac{|g'_n(0)|}{1+|g_n(0)|^2}=1.
\end{equation}
By \eqref{eq2.60} we see that $g'(0)\neq 0$ and so $g'\not\equiv 0.$ Therefore, $g$ is a non-constant meromorphic function in the complex plane. Moreover, for each
$\zeta\in\Bbb{C},$ there exists a positive number $R,$ such that $\zeta\in \overline{D}_R(0).$ Then, it follows by \eqref{eq2.52} that there exists some positive integer, say $N_R$ such that for $n>N_R$ we have $\zeta\in \overline{D}_R(0) \subset D_{R_n}(0).$
Combining this with \eqref{eq2.59} and the right equality of \eqref{eq2.58}, we derive
\begin{equation}\nonumber
g^{\#}(\zeta)=\frac{|g'(\zeta)|}{1+|g(\zeta)|^2}=\lim\limits_{n\rightarrow\infty}\frac{|g'_n(\zeta)|}{1+|g_n(\zeta)|^2}\leq \lim\limits_{n\rightarrow\infty}\frac{(1+\varepsilon_n)^{1+\alpha}}{(1-\varepsilon_n)^{2\alpha}}=g^{\#}(0)=1
\end{equation}
for each given $\zeta\in\Bbb{C}.$ This proves Lemma \ref{lemma2.15} for $0\leq \alpha<1.$
\vskip 2mm
\par {\bf (III)} From the proof of Lemma \ref{lemma2.15} for $0\leq\alpha<1$ in Remark \ref{remark2.2} (II) above, we see that $r_n$ and $r^{*}_n$ can be chosen as
$r_n=\frac{1}{n^{m}}$ and $r^{*}_n=\frac{1}{pn^{m}}$ for $n\in\Bbb{Z}^{+}\setminus\{1\}$ respectively,
where $m$ and $p$ are any given positive integers with $p\geq 2.$ Next, we use the lines of the proof of Lemma \ref{lemma2.15} for $0\leq\alpha<1$ in Remark \ref{remark2.2} (II), we also get the conclusion of Lemma \ref{lemma2.15} for $0\leq\alpha<1.$
\end{remark}
\vskip 2mm
\par From Remark \ref{remark2.2} (III) and the proof of Lemma \ref{lemma2.15} for $0\leq\alpha<1$ in Remark \ref{remark2.2} (II), we get the following result:
\vskip 2mm
\par
\begin{lemma}\rm{}\label{lemma2.16}
\textit{Let $\left\{f_m(z)\right\}_{m=1}^{\infty}$ be a family of holomorphic functions in a disc $D_{\delta_m}(z_0)=\{z\in\Bbb{C}:|z-z_0|<\delta_m\}\subset D_0,$ where $z_0\in D_0$ is a point, $D_0\subseteq\Bbb{C}$ is a domain in the complex plane, and $\{\delta_m\}_{m=1}^{\infty}$ is an infinite sequence of positive numbers such that $\delta_{m+1}\leq \delta_{m}$ for each positive integer $m\in\Bbb{Z}^{+},$ and let $\alpha$ be a real number
satisfying $-1<\alpha<1.$ Then, if $\{f_m(z)\}_{m=1}^{\infty}$ is not normal at the point
$z_0,$ there exist, for each $-1 <\alpha < 1$:
}
\vskip 2mm
\par (i) \,  \textit{ points $z_m\in D_{\delta_m}(z_0) ,$ $z_m\rightarrow z_0,$}
\vskip 2mm
\par (ii) \, \textit{positive numbers $\rho_m,$ $\rho_m\rightarrow 0^{+}$}
and
\vskip 2mm
\par (iii) \, \textit{an infinite subsequence of the family $\left\{f_m(z)\right\}_{m=1}^{\infty},$ say $\left\{f_m(z)\right\}_{m=1}^{\infty}$ itself, such that $\frac{f_m(z_n+\rho_m\zeta)}{\rho_m^{\alpha}}\rightarrow
h(\zeta)$ spherically uniformly on compact subset of $\Bbb{C},$
where $h$ is a non-constant entire function. The function $h$
may be taken to satisfy the normalization $h^{\#}(\zeta)\leq
h^{\#}(0)=1.$
}
\end{lemma}
\vskip 2mm
\par Next we follow Bergweiler \cite[p.153]{Bergweiler1993} to introduce some definitions and results in the theory of iterations of the meromorphic functions in the complex plane that play an important role in proving the main results in this paper: let  $f : \Bbb{C}\rightarrow\Bbb{C}\cup\{\infty\}$ be a meromorphic
function, where $\Bbb{C}$ is the complex plane, and we assume that $f$ is neither a constant nor a linear
transformation. We denote by  $f^{\circ n}$ the $n$-th iterate of $f,$ that is, $f^{\circ 0}(z) = z$ with $z\in\Bbb{C}$  and
$f^{\circ n}(z)= f(f^{\circ n-1}(z))$ with $z\in\Bbb{C}$ for $n\in\Bbb{Z}^{+}$. Then $f^{\circ n}(z)$ is defined for all $z \in\Bbb{C}$ except
for a countable set which consists of the poles of $f^{\circ 1},$ $f^{\circ 2},$ $\ldots,$ $f^{\circ n-1}.$ If $f$ is
rational, then $f$ has a meromorphic extension to $\Bbb{C}\cup\{\infty\},$ and we denote the extension
again by $f,$ we see that $f^{\circ n}$ is defined and meromorphic in $\Bbb{C}\cup\{\infty\}.$ But if $f$
is transcendental in the complex plane, it is not reasonable for us to define $f(\infty).$
\vskip 2mm
\par The basic objects studied in iteration theory are the Fatou set $F = F(f)$ and
the Julia set $J=J(f)$ of a meromorphic function $f.$ Roughly speaking, the
Fatou set is the set where the iterative behavior is relatively tame in the sense
that points close to each other behave similarly, while the Julia set is the set where chaotic phenomena take place. The formal definitions are
\begin{equation}\nonumber
 F=\left \{z \in \hat{\Bbb{C}}: \, \left\{f^{\circ n}\right\}_{n=1}^{+\infty} \, \, \text{is defined and normal in some neighborhood of} \, \, z\right\}
\end{equation}
and $J=J(f)=\hat{\Bbb{C}}\setminus F,$ where and in what follows, $\hat{\Bbb{C}}=\Bbb{C}\cup\{\infty\}.$  We also need the following  notion of the periodic point in iteration theory for proving the main results in this paper:
\begin{definition}\rm{}
 By definition, $z_0\in\Bbb{C}$ is called a periodic point of a non-constant meromorphic function $f$ if $f^{\circ n}(z_0) = z_0$ for
some positive integer $n$ satisfying $n \geq 1.$ In this case, the positive integer $n$ is called a period of the point $z_0,$ and the smallest positive integer $n$ with this property is called the minimal period of the point $z_0.$ For a periodic point $z_ 0\in\Bbb{C}$ 0f
minimal period $n,$  $(f^{\circ n})'(z_0)$ is called the multiplier of the point $z_0.$ We mention that $z_0 =\infty,$ which
can happen only for the rational function $f,$ of course, this has to be modified. In this case, the multiplier is defined to be $(g^{\circ n})'(0)$
where $g(z) = 1/f(1/z).$ A periodic point is called attracting, indifferent, or repelling accordingly as the
modulus of its multiplier is less than, equal to, or greater than $1.$ Periodic
points of multiplier $0$ are called super attracting. Some writers reserve the
term attracting for the case $0< \left|(f^{\circ n})'(z_0)\right| < 1,$ but we consider super attracting
as a special case of attracting. The multiplier of an indifferent periodic point is of the form $e^{2\pi\alpha}$ where $0 < \alpha < 1.$ We say that $z_0$ is rationally indifferent if $\alpha$ is rational and irrationally indifferent otherwise. Also, a point $z_0$ is called
pre-periodic if $f^{\circ n}(z_0)$ is periodic for some $n \geq 1.$ Finally, a periodic point of
period $1$ is called a fixed point. It is easy to see that attracting periodic points are in $F,$ while repelling
and rationally indifferent periodic points are in $J.$ For irrationally indifferent
periodic points the question whether they are in $F$ or $J$ is difficult to decide.
Both possibilities do occur. We refer the reader to the classical papers of Cremer
 \cite{Cremer1928,Cremer1938} and Siegel \cite{Siegel1942}, as well as more recent work of Yoccoz \cite{Yoccoz1988,Yoccoz1995}. An
exposition of these and other results, together with further references, can be
found in \cite{Marco1992}. The behavior of the iterates in the neighborhood of a fixed point (or, more
generally, a periodic point) is intimately connected with the solution of certain
functional equations. For most results in this direction, it is required only that
the function under consideration is defined in a neighborhood of the fixed point,
and it is usually irrelevant whether it extends to a rational or transcendental
meromorphic function. Therefore, we omit this topic here but refer to the
p`apers and books on iteration of rational functions cited in the introduction.
\end{definition}
\vskip 2mm
\par We recall the following results from Bergweiler \cite[p.153]{Bergweiler1993}:
\vskip 2mm
\par
\begin{lemma}\rm{}(\cite[p.161, Theorem 5]{Bergweiler1993})\label{lemma2.17} \textit{If $f$ is a transcendental meromorphic function in the complex plane and $n$ is a positive integer such that $n\geq 2,$ then $f$ has infinitely many repelling periodic points of minimal period $n.$}
\end{lemma}
\vskip 2mm
\par
\begin{lemma}\rm{}(\cite[p.160, Theorem 4]{Bergweiler1993})\label{lemma2.18}\textit{ Let f be a non-constant  meromorphic function in the complex plane. Then, the Julia set $J=J(f)$ of the  meromorphic function $f$ is the closure of the set of repelling periodic points of f. }
\end{lemma}
\vskip 2mm
\par The following result is due to Clunie \cite{Clunie1970}:
\vskip 2mm
\par
\begin{lemma}\rm{}(\cite{Clunie1970})\label{lemma2.19}\textit{ Let  $g$ be a transcendental meromorphic function in the complex plane, and let $h$ be a non-constant entire function. Then 
\begin{equation}\nonumber
\lim\limits_{r\rightarrow\infty}\frac{T(r,g(h))}{T(r,h)}=\infty.
\end{equation}
}
\end{lemma}
\vskip 2mm
\par
\maketitle
\section{Notations and definitions for the proofs}
Before we prove the main results of this paper,  we first introduce the following notations and definitions: Let $f$ and $g$ be two non-constant meromorphic functions in the complex plane, and let $a$ be a value in the extended plane. Let
$\overline{N}_{E}(r,a)$ ``count'' those points in $\overline{N}(r,
1/(f-a)),$ where $a$ is taken by $f$ and $g$ with the same
multiplicity, and each point is counted only once, and let
$\overline{N}_0(r, a)$ be the reduced counting function of the
common $a$-points of $f$ and $g$ in $\overline{N}(r, 1/(f-a)),$
where $\overline{N}(r,1/(f-\infty))$ means $\overline{N}(r,f).$ We
say that $f$ and $g$ share the value $a$ CM$^*$, if
$$\overline{N}\left(r,\frac{1}{f-a}\right)-\overline{N}_{E}(r,a)=S(r,f) \quad \text{and} \quad
\overline{N}\left(r,\frac{1}{g-a}\right)-\overline{N}_{E}(r,a)=S(r,g).$$
These definitions can be found in \cite{YangYi2003}, where CM$^*$ here is denoted by ``CM'' in \cite{YangYi2003}. We also need the following definition:
\vskip 2mm
\par
\begin{definition}\rm{}(\cite[Definition 1]{Lahiri2001}) Let $p$ be a positive integer and $a\in \Bbb{C}\bigcup\{\infty\}.$ Next we denote by $N_{p)}\left(r,\frac{1}{f-a}\right)$ the counting function of those $a$-points of $f$(counted with proper multiplicities) whose multiplicities are not greater than $p,$ and denote by $N_{(p}\left(r,\frac{1}{f-a}\right)$ the counting function of those $a$-points of $f$ (counted with proper multiplicities) whose multiplicities are not less than $p.$ We denote by  $\overline{N}_{p)}\left(r,\frac{1}{f-a}\right)$ and  $\overline{N}_{(p}\left(r,\frac{1}{f-a}\right)$ the reduced forms of $N_{p)}\left(r,\frac{1}{f-a}\right)$ and $N_{(p}\left(r,\frac{1}{f-a}\right)$ respectively. Here  $N_{p)}\left(r,\frac{1}{f-\infty}\right),$ $\overline{N}_{p)}\left(r,\frac{1}{f-\infty}\right),$ $N_{(p}\left(r,\frac{1}{f-\infty}\right)$ and
$\overline{N}_{(p}\left(r,\frac{1}{f-\infty}\right)$ mean $N_{p)}\left(r, f\right),$ $\overline{N}_{p)}\left(r, f\right),$ $N_{(p}\left(r, f\right)$ and $\overline{N}_{(p}\left(r, f\right)$ respectively.
\end{definition}
\vskip 2mm
\par
\begin{definition}\rm{(\cite[Definition 2.1]{GundersenHayman2004})}
For a meromorphic function $f$ satisfying $f\not\equiv 0$ and a positive
integer $j,$ let $n_j\left(r, \frac{1}{f}\right)$ denote the number of zeros of $f$ in $\{z : |z|< r\},$ counted in
the following manner: a zero of $f$ of multiplicity $m$ is counted exactly $k$ times where
$k = \min\{ m, j\}.$ Then let $N_j\left(r, \frac{1} {f}\right)$ denote the corresponding integrated counting
function; that is
\begin{equation}\nonumber
N_j\left(r, \frac{1} {f}\right)=n_j\left(0, \frac{1} {f}\right)\log r+\int_0^r\frac{n_j\left(t,\frac{1}{f}\right)-n_j\left(0, \frac{1} {f}\right)}{t}dt.
\end{equation}
Regarding the well-known integrated counting functions $\overline{N}\left(r, \frac{1}{f}\right)$ and $N\left(r, \frac{1}{f}\right),$
we see that $\overline{N}\left(r, \frac{1}{f}\right)\leq N_j\left(r, \frac{1}{f}\right) \leq N\left(r, \frac{1}{f}\right),$ $N_j\left(r, \frac{1}{ f}\right) \leq  j\overline{N}\left(r, \frac{1}{ f}\right)$ and $N_1\left(r, \frac{1}{f}\right) = \overline{N}\left(r, \frac{1}{f}\right).$
\end{definition}
\vskip 2mm
\par Secondly, we need generalizations of previous definitions. Next we Let $f_1,$ $f_2,$ $\cdots,$ $f_k$ be $k$ non-constant meromorphic functions. We denote by $\overline{N}_0(r, f_1, f_2, \cdots, f_k)$ the reduced counting function of common poles of $f_1,$ $f_2,$ $\cdots,$ $f_k$ in $|z|<r$ and denote by
$\overline{N}_0\left(r, \frac{1}{f_1}, \frac{1}{f_2}, \cdots, \frac{1}{f_k}\right)$ the reduced counting function of common zeros of  $f_1,$ $f_2,$ $\cdots,$ $f_k$ in $|z|<r.$ Restrictions can be given in these definitions. For instance, the notation $\overline{N}_0\left(r,\frac{1}{f},\frac{1}{g'}, \frac{1}{h} |g\neq 0\right)$ will mean the reduced counting function of common zeros of $f,$ $g',$ and $h$ in $|z|<r$ that are not zeros of $g.$ The notation $\overline{N}_0\left(r, \frac{1}{h'}, \frac{1}{g}\left.\right| h\neq 0, f\neq 0\right)$ will mean the reduced counting function of common zeros of $h'$ and $g$ in $|z|<r$ that are not zeros of $h$ and $f.$ The notation $\overline{N}_0\left(r, \frac{1}{f_1}, \frac{1}{f_2}, \frac{1}{f_3}, \frac{1}{f_3}\left.\right| h= 0 \, \, \text{or} \, \, f= 0\right)$ will mean the reduced counting function of common zeros of $f_1,$ $f_2,$ $f_3,$ $f_4$ in $|z|<r$ that are either zeros of $h$ or zeros of $f.$
\vskip 2mm
\par Thirdly, we denote by $\overline{N}_{(1,1,1)}(r,f_1, f_2, f_3)$ the reduced counting function of common simple poles of $f_1,$ $f_2$ and $f_3$ in $|z|<r,$ and denote by $\overline{N}_{(1,1)}\left(r,\frac{1}{f_1},\frac{1}{f_2}\right)$ the reduced counting function of common simple zeros of $f_1$ and $f_2$ in $|z|<r.$
\vskip 2mm
\par
\section{\bf Proof of Theorem \ref{Theorem1.10}}
 Suppose that there exist three transcendental meromorphic functions $f,$ $g$ and $h$
satisfying \eqref{eq1.1} for $n=8.$
\vskip 2mm
\par We consider the following two cases:
\vskip 2mm
\par{\bf Case 1.} \, Suppose that  $f^8,$ $g^8$ and $h^8$ are linearly dependent. Then there exist three constants $c_1,$ $c_2,$  $c_3$ satisfying $|c_1|+|c_2|+|c_3|>0$ such that
\begin{equation}\label{eq4.1}
c_1f^8+c_2g^8+c_3h^8=0.
\end{equation}
Without loss of generality, we suppose that $c_3\neq 0.$ Then, by substituting \eqref{eq4.1} into $f^8+g^8+h^8=1,$ we have
\begin{equation}\label{eq4.2}
\left(1-\frac{c_1}{c_3}\right)f^8+\left(1-\frac{c_2}{c_3}\right)g^8=1.
\end{equation}
By \eqref{eq4.2} and the well-known result from Gross \cite{Gross1966} that the equation $u^8+v^8=1$ cannot have non-constant meromorphic solutions $u$ and $v,$ we get a contradiction.
\vskip 2mm
\par{\bf Case 2.} \, Suppose that $f^8,$ $g^8$ and $h^8$ are linearly independent. First of all, by Lemma \ref{lemma2.11} we can deduce that $f,$ $g$ and $h$ share $0$ and $\infty$ CM* such that
\begin{equation}\label{eq4.3}
\begin{split}
\overline{N}_{1)}(r,f)&=\overline{N}_{1)}(r,g)+S(r)=\overline{N}_{1)}(r,h)+S(r)=\frac{1}{3}T(r)+S(r)\\
                      &=T(r,f)+S(r)=T(r,g)+S(r)=T(r,h)+S(r)\\
                      &=\overline{N}_0(r,f,g,h)+S(r)
\end{split}
\end{equation}
and
\begin{equation}\label{eq4.4}
\begin{split}
\overline{N}_{1)}\left(r,\frac{1}{f}\right)&=\overline{N}_{1)}\left(r,\frac{1}{g}\right)+S(r)
=\overline{N}_{1)}\left(r,\frac{1}{h}\right)+S(r)=\frac{1}{3}T(r)+S(r)\\
 &=T(r,f)+S(r)=T(r,g)+S(r)=T(r,h)+S(r),
\end{split}
\end{equation}
where and in what follows, $T(r)$ and $S(r)$ are defined as in Lemma \ref{lemma2.2}.
Also by Lemma \ref{lemma2.11} we have
\begin{equation}\label{eq4.5}
\overline{N}_0\left(r,\frac{1}{f}, \frac{1}{g}\right)+\overline{N}_0\left(r,\frac{1}{g}, \frac{1}{h}\right)+\overline{N}_0\left(r,\frac{1}{h}, \frac{1}{f}\right)=S(r),
\end{equation}
which means that any two of the functions $f,$ $g$ and $h$ almost have no common zeros. Moreover, by \eqref{eq4.3}-\eqref{eq4.5} we deduce
\begin{equation}\label{eq4.6}
 \begin{split}
 \frac{1}{3}T(r)&=T(r,f)+S(r)=T(r,g)+S(r)=T(r,h)+S(r)\\
&=\overline{N}_{1)}\left(r,\frac{1}{f}\right)+S(r)=\overline{N}_{1)}\left(r,\frac{1}{g}\right)+S(r)\\
&=\overline{N}_{1)}\left(r,\frac{1}{h}\right)+S(r)=\overline{N}_{1)}(r,f)+S(r)\\
&=\overline{N}_{1)}(r,g)+S_1(r)=\overline{N}_{1)}(r,h)+S(r)\\
 &=\overline{N}_{(1,1,1)}(r, f, g, h)+S(r)\\
 \end{split}
 \end{equation}
 and
\begin{equation}\label{eq4.7}
 \overline{N}_{(1,1)}\left(r,\frac{1}{f}, \frac{1}{g}\right)+\left(r,\frac{1}{g}, \frac{1}{h}\right)
  +\overline{N}_{(1,1)}\left(r,\frac{1}{h}, \frac{1}{f}\right)=S(r).
\end{equation}
\vskip 2mm
\par On the other hand, by \eqref{eq2.1}, \eqref{eq2.7} and Theorem \ref{TheoremA} we have for $n=m=k=8$ that
\begin{equation} \label {eq4.8}
 W(f^8, g^8, h^8)=\left|\aligned (f^8)' &\quad  (g^8)'\\
                                   (f^8)'' &\quad  (g^8)''
                                   \endaligned
                                  \right|= \left|\aligned (g^8)' &\quad  (h^8)'\\
                                  (g^8)'' &\quad  (h^8)''\endaligned\right| = \left|\aligned (h^8)' &\quad  (f^8)'\\
                                  (h^8)'' &\quad  (f^8)''\endaligned\right|= \alpha f^6g^6h^6,
\end{equation}
where $\alpha$ is a small function of $f,$ $g$ and $h$ such that $\alpha\not\equiv 0, \infty.$
\vskip 2mm
\par On the other hand, by taking the first derivatives of two sides of
\eqref{eq1.1} for $n=8$ we have
\begin{equation}\label {eq4.9}
f^7f'+g^7g'+h^7h'=0,
\end{equation}
and so we have
\begin{equation}\label{eq4.10}
g^7g'\left(\frac{f^7f'}{g^7g'}+1\right)=-h^7h'
\end{equation}
and
\begin{equation}\label {eq4.11}
\frac{f^7f'}{g^7g'}+\frac{h^7h'}{g^7g'}+1=0.
\end{equation}
Now we prove the following claim:
\vskip 2mm
\par {\bf  Claim I.} The following items hold: (i) $\frac{f^7f'}{g^7g'}+1$ and $-h^7h'$ share $0$ CM*; (ii) $\frac{g^7g'}{h^7h'}+1$ and $-f^7f'$ share $0$ CM*; (iii) $\frac{h^7h'}{f^7f'}+1$ and $-g^7g'$ share $0$ CM*.
\vskip 2mm
\par Without loss of generality, next we only prove Claim I (i): First of all, by \eqref{eq4.3} we have
\begin{equation}\label {eq4.12}
\overline{N}\left(r, -h^7h'\right)=\overline{N}_0\left(r, -h^7h', g^7g'\right)+S(r)=\overline{N}_0\left(r, h, g\right)+S(r)
=T(r)+S(r).
\end{equation}
 By \eqref{eq4.12} we have
\begin{equation}\label {eq4.13}
\overline{N}_0\left(r, \frac{1}{-h^7h'}, g^7g'\right)=S(r).
\end{equation}
Next we prove
\begin{equation}\label {eq4.14}
\overline{N}_0\left(r, \frac{1}{-h^7h'}, \frac{1}{g^7g'}\right)=S(r).
\end{equation}
For this purpose, now we let $z_0$ be a common zero of $-h^7h'$ and $g^7g'$ in the complex plane. Then at least one of the following four cases must hold: $h'(z_0)=g(z_0)=0,$ $h'(z_0)=g'(z_0)=0,$  $h(z_0)=g(z_0)=0$ or $h(z_0)=g'(z_0)=0.$ We discuss these four cases as follows:
\vskip 2mm
\par {\bf  (1).}  Suppose that $h'(z_0)=g(z_0)=0.$ Then, by \eqref{eq4.8} we can deduce that the multiplicity of $z_0$ as a zero of
\begin{equation}\label {eq4.15}
 \left|\aligned (g^8)' &\quad  (h^8)'\\
                                  (g^8)'' &\quad  (h^8)''\endaligned\right|=(g^8)' (h^8)''-(h^8)' (g^8)''= \alpha f^6g^6h^6
   \end{equation}
is not less than $7.$ We discuss this as follows:
\vskip 2mm
\par Suppose that $h(z_0)\neq 0$ and  $f(z_0)\neq 0.$ Then, by comparing the multiplicities of $z_0$ as the zeros of two sides of  \eqref{eq4.15} we deduce that $z_0$ is a zero of $\alpha.$
Therefore,
\begin{equation}\label {eq4.16}
\overline{N}_0\left(r, \frac{1}{h'}, \frac{1}{g}\left.\right| h\neq 0, f\neq 0\right)\leq N\left(r,\frac{1}{\alpha}\right)\leq T(r,\alpha)+O(1)=S(r),
   \end{equation}
and so
\begin{equation}\label {eq4.17}
\overline{N}_0\left(r, \frac{1}{-h^7h'}, \frac{1}{g^7g'}, \frac{1}{h'}, \frac{1}{g}\left.\right| h\neq 0, f\neq 0\right)=S(r).
   \end{equation}
 \vskip 2mm
\par Suppose that $h(z_0)=0$ or $f(z_0)=0.$ Then, by the supposition $h'(z_0)=g(z_0)=0$ we can see that $z_0$ is a common zero of $h$ and $g,$ or a common zero of $f$ and $g.$ Moreover,  by \eqref{eq4.4} and \eqref{eq4.5} we have
\begin{equation}\label {eq4.18}
\overline{N}_0\left(r, \frac{1}{h'}, \frac{1}{g},\frac{1} {h} \left.\right|h= 0 \, \, \text{or} \, \, f= 0\right)=S(r)
   \end{equation}
and
 \begin{equation}\label {eq4.19}
\overline{N}_0\left(r, \frac{1}{h'}, \frac{1}{g}, \frac{1}{f}\left.\right| h= 0 \, \, \text{or} \, \, f= 0 \right)=S(r),
   \end{equation}
and so
\begin{equation}\label {eq4.20}
\begin{split}
&\quad \overline{N}_0\left(r, \frac{1}{-h^7h'}, \frac{1}{g^7g'}, \frac{1}{h'}, \frac{1}{g}\left.\right| h= 0 \, \, \text{or} \, \, f= 0\right)\\
&\leq\overline{N}_0\left(r, \frac{1}{h^7h'}, \frac{1}{g^7g'}, \frac{1}{h'}, \frac{1}{g},\frac{1}{h}\left.\right|f= 0\right)
+\overline{N}_0\left(r, \frac{1}{h^7h'}, \frac{1}{g^7g'}, \frac{1}{h'}, \frac{1}{g},\frac{1}{f}\left.\right| h= 0\right)\\
\end{split}
\end{equation}
\begin{equation}\nonumber
\begin{split}
&\leq \overline{N}_0\left(r,\frac{1}{g}, \frac{1}{h}\right)+\overline{N}_0\left(r,\frac{1}{f}, \frac{1}{g}\right)+\overline{N}_0\left(r,\frac{1}{h}, \frac{1}{f}\right)=S(r).
  \end{split}
   \end{equation}
\vskip 2mm
\par By \eqref{eq4.17} and \eqref{eq4.20} we have
\begin{equation}\label {eq4.21}
\begin{split}
 \overline{N}_0\left(r, \frac{1}{-h^7h'}, \frac{1}{g^7g'}, \frac{1}{h'}, \frac{1}{g}\right)
&\leq \overline{N}_0\left(r, \frac{1}{-h^7h'}, \frac{1}{g^7g'}, \frac{1}{h'}, \frac{1}{g}\left.\right| h\neq 0, f\neq 0\right)\\
&\quad+\overline{N}_0\left(r, \frac{1}{-h^7h'}, \frac{1}{g^7g'}, \frac{1}{h'}, \frac{1}{g}\left.\right| h= 0 \, \, \text{or} \, \, f= 0\right)\\
&=S(r).
  \end{split}
   \end{equation}
\vskip 2mm
\par {\bf  (2).}  Suppose that  $h'(z_0)=g'(z_0)=0.$ Then, by \eqref{eq4.9} we have $f^7(z_0)f'(z_0)=0,$ and so either $f(z_0)=0$ or $f'(z_0)=0.$  We discuss this as follows:
\vskip 2mm
\par Suppose that $f(z_0)=0.$ Then, by \eqref{eq4.8} and the supposition  $h'(z_0)=g'(z_0)=0,$ we deduce that the multiplicity of $z_0$ as a zero of
\begin{equation}\label {eq4.22}
\left|\aligned (f^8)' &\quad  (g^8)'\\
 (f^8)'' &\quad  (g^8)''
                                   \endaligned\right|=(f^8)'(g^8)''-(f^8)''(g^8)'= \alpha f^6g^6h^6
    \end{equation}
 and
     \begin{equation}\label {eq4.23}
     \left|\aligned (h^8)' &\quad  (f^8)'\\
                                  (h^8)'' &\quad  (f^8)''\endaligned\right|=(h^8)' (f^8)''-(h^8)'' (f^8)'= \alpha f^6g^6h^6
   \end{equation}
 is not less than $7.$ This implies that, possibly $z_0$ is a multiple zero of $f,$ or a zero of one of $g$ and $h,$ or a zero of $\alpha.$  Therefore, by \eqref{eq4.4} and \eqref{eq4.5}, we deduce
\begin{equation}\nonumber
\begin{split}
\overline{N}_0\left(r,\frac{1}{h'}, \frac{1}{g'},\frac{1}{f}\right)&\leq \overline {N}_{(2}\left(r,\frac{1}{f}\right)+ \overline{N}_0\left(r,\frac{1}{f},\frac{1}{g}\right)+\overline{N}_0\left(r,\frac{1}{h},\frac{1}{f}\right)+N\left(r,\frac{1}{\alpha}\right)\\
&\leq T(r,\alpha)+S(r)=S(r),
 \end{split}
  \end{equation}
and so
\begin{equation}\label {eq4.24}
\overline{N}_0\left(r,\frac{1}{-h^7h'}, \frac{1}{g^7g'},\frac{1}{h'}, \frac{1}{g'},\frac{1}{f}\right)\leq \overline{N}_0\left(r,\frac{1}{h'}, \frac{1}{g'},\frac{1}{f}\right)= S(r).
 \end{equation}
 \vskip 2mm
\par Suppose that $f'(z_0)=0.$ Then, we deduce by \eqref{eq4.8} and the supposition $g'(z_0)=h'(z_0)=0$ that possibly $z_0$ is a multiple zero of one of $f,$ $g$ and $h,$ or a zero of $\alpha.$ Therefore, by \eqref{eq4.4} we deduce
\begin{equation}\nonumber
\begin{split}
\overline{N}_0\left(r,\frac{1}{h'}, \frac{1}{g'},\frac{1}{f'}\right)
&\leq \overline{N}_0\left(r,\frac{1}{f},\frac{1}{f'}\right)+\overline{N}_0\left(r,\frac{1}{g},\frac{1}{g'}\right)
+\overline{N}_0\left(r,\frac{1}{h},\frac{1}{h'}\right)+N\left(r,\frac{1}{\alpha}\right)\\
&\leq\overline{N}_{(2}\left(r,\frac{1}{f}\right)+\overline{N}_{(2}\left(r,\frac{1}{g}\right)+\overline{N}_{(2}\left(r,\frac{1}{h}\right)+T(r,\alpha)+O(1)\\
&=S(r),
 \end{split}
 \end{equation}
 and so
 \begin{equation}\label {eq4.25}
\begin{split}
\overline{N}_0\left(r,\frac{1}{-h^7h'}, \frac{1}{g^7g'}, \frac{1}{h'}, \frac{1}{g'},\frac{1}{f'}\right)&\leq \overline{N}_0\left(r,\frac{1}{h'}, \frac{1}{g'},\frac{1}{f'}\right)=S(r).
 \end{split}
 \end{equation}
\vskip 2mm
\par
By \eqref{eq4.24}, \eqref{eq4.25} and the above analysis we have
\begin{equation}\label{eq4.26}
\begin{split}
\overline{N}_0\left(r,\frac{1}{-h^7h'}, \frac{1}{g^7g'},\frac{1}{h'}, \frac{1}{g'}\right)
&\leq \overline{N}_0\left(r,\frac{1}{-h^7h'}, \frac{1}{g^7g'},\frac{1}{h'}, \frac{1}{g'},\frac{1}{f}\right)\\
&\quad+\overline{N}_0\left(r,\frac{1}{-h^7h'}, \frac{1}{g^7g'}, \frac{1}{h'}, \frac{1}{g'},\frac{1}{f'}\right)=S(r).
  \end{split}
 \end{equation}
\vskip 2mm
\par {\bf  (3).} Suppose that  $h(z_0)=g(z_0)=0.$ Then from \eqref{eq4.5} we have
 \begin{equation}\label{eq4.27}
\overline{N}_0\left(r, \frac{1}{-h^7h'}, \frac{1}{g^7g'}, \frac{1}{h}, \frac{1}{g} \right)=S(r).
\end{equation}
 \vskip 2mm
\par {\bf  (4).} Suppose that $h(z_0)=g'(z_0)=0$ holds. Then, in the same manner as in the above case (1) we deduce by \eqref{eq4.4}, \eqref{eq4.5} and \eqref{eq4.8} that
\begin{equation}\label {eq4.28}
\overline{N}_0\left(r, \frac{1}{-h^7h'}, \frac{1}{g^7g'}, \frac{1}{h}, \frac{1}{g'}\left.\right| g\neq 0, f\neq 0\right)\leq \overline{N}_0\left(r, \frac{1}{g'}, \frac{1}{h}\left.\right| g\neq 0, f\neq 0\right)=S(r)
   \end{equation}
and
\begin{equation}\label{eq4.29}
\begin{split}
&\quad \overline{N}_0\left(r, \frac{1}{-h^7h'}, \frac{1}{g^7g'}, \frac{1}{h}, \frac{1}{g'}\left.\right| g= 0 \, \, \text{or} \, \, f= 0\right)\\
\end{split}
\end{equation}
\begin{equation}\nonumber
\begin{split}
&\leq\overline{N}_0\left(r, \frac{1}{h^7h'}, \frac{1}{g^7g'},\frac{1}{h}, \frac{1}{g'}, \frac{1}{g}\left.\right|f\neq 0\right)
+\overline{N}_0\left(r, \frac{1}{h^7h'}, \frac{1}{g^7g'}, \frac{1}{h}, \frac{1}{g'},\frac{1}{f}\left.\right| g\neq 0\right)\\
 &\quad+\overline{N}_0\left(r,\frac{1}{f}, \frac{1}{g}\right)
\leq \overline{N}_0\left(r,\frac{1}{g}, \frac{1}{h}\right)+\overline{N}_0\left(r,\frac{1}{h}, \frac{1}{f}\right)+\overline{N}_0\left(r,\frac{1}{f}, \frac{1}{g}\right)=S(r).
  \end{split}
   \end{equation}
By \eqref{eq4.28} and \eqref{eq4.29} we have
\begin{equation}\label{eq4.30}
\begin{split}
\overline{N}_0\left(r, \frac{1}{-h^7h'}, \frac{1}{g^7g'}, \frac{1}{h}, \frac{1}{g'}\right)&\leq
\overline{N}_0\left(r, \frac{1}{-h^7h'} , \frac{1}{g^7g'}, \frac{1}{h}, \frac{1}{g'}\left.\right| g\neq 0, f\neq 0\right)\\
&\quad +\overline{N}_0\left(r, \frac{1}{-h^7h'}, \frac{1}{g^7g'}, \frac{1}{h}, \frac{1}{g'}\left.\right| g= 0 \, \, \text{or} \, \, f= 0\right)\\
&=S(r).
\end{split}
\end{equation}
Finally, by \eqref{eq4.21}, \eqref{eq4.26}, \eqref{eq4.27}, \eqref{eq4.30} and the above analysis in cases (1)-(4) we have
\begin{equation}\label{eq4.31}
\begin{split}
\overline{N}_0\left(r, \frac{1}{-h^7h'}, \frac{1}{g^7g'}\right)
&\leq \overline{N}_0\left(r, \frac{1}{-h^7h'}, \frac{1}{g^7g'}, \frac{1}{h'}, \frac{1}{g}\right)+\overline{N}_0\left(r,\frac{1}{-h^7h'}, \frac{1}{g^7g'},\frac{1}{h'}, \frac{1}{g'}\right)\\
&\quad +\overline{N}_0\left(r, \frac{1}{-h^7h'}, \frac{1}{g^7g'}, \frac{1}{h}, \frac{1}{g} \right)+\overline{N}_0\left(r, \frac{1}{-h^7h'}, \frac{1}{g^7g'}, \frac{1}{h}, \frac{1}{g'}\right)\\
&=S(r),
\end{split}
\end{equation}
which is \eqref{eq4.14}. By \eqref{eq4.10}, \eqref{eq4.13} and \eqref{eq4.14} we have Claim I(i). Next we prove the following claim:
\vskip 2mm
\par {\bf  Claim II.} \, The following equalities hold:
\begin{equation}\nonumber
\begin{split}
\rm{(i)}\quad &T\left(r,\frac{f^7f'}{g^7g'}\right)=\overline{N}\left(r,\frac{f^7f'}{g^7g'}\right)+\overline{N}\left(r,\frac{g^7g'}{f^7f'}\right)+\overline{N}\left(r, \frac{1}{\frac{f^7f'}{g^7g'}+1}\right)+S(r),\\
\rm{(ii)}\quad &T\left(r,\frac{g^7g'}{h^7h'}\right)
=\overline{N}\left(r,\frac{g^7g'}{h^7h'}\right)+\overline{N}\left(r,\frac{h^7h'}{g^7g'}\right)+\overline{N}\left(r,\frac{1}{\frac{g^7g'}{h^7h'}+1}\right)+S(r),\\
\rm{(iii)} \quad &T\left(r,\frac{h^7h'}{f^7f'}\right)=\overline{N}\left(r,\frac{h^7h'}{f^7f'}\right)+\overline{N}\left(r,\frac{f^7f'}{h^7h'}\right)+\overline{N}\left(r,\frac{1}{\frac{h^7h'}{f^7f'}+1}\right)+S(r),\\
\rm{(iv)} \quad  &3T(r) =T\left(r,\frac{f^7f'}{g^7g'}\right)+S(r)=T\left(r,\frac{g^7g'}{h^7h'}\right)+S(r)=T\left(r,\frac{h^7h'}{f^7f'}\right)+S(r),\\
 \rm{(v)}\quad &2T(r,f)=N\left(r,\frac{1}{f'}\right)+S(r)=N_0\left(r,\frac{1}{f'}\right)+S(r)= \overline{N}_0\left(r,\frac{1}{f'}\right)+S(r),\\
\rm{(vi)}\quad  &2T(r,g)=N \left(r,\frac{1}{g'}\right)+S(r)=N_0\left(r,\frac{1}{g'}\right)+S(r)= \overline{N}_0\left(r,\frac{1}{g'}\right)+S(r),\\
\rm{(vii)}\quad &2T(r,h)=N\left(r,\frac{1}{h'}\right)=N_0\left(r,\frac{1}{h'}\right)+S(r)= \overline{N}_0\left(r,\frac{1}{h'}\right)+S(r),
 \end{split}
  \end{equation}
where and in what follows, $N_0\left(r,\frac{1}{f'}\right)$ denotes the counting function of those zeros in $N\left(r,\frac{1}{f'}\right)$ that are not zeros of $f.$  $N_0\left(r,\frac{1}{g'}\right)$ and $N_0\left(r,\frac{1}{h'}\right)$ have similar definitions.
\vskip 2mm
\par Indeed, by Claim I(i) we know that $\frac{f^7f}{g^7g'}+1$ and $-h^7h'$ share $0$ CM*, and by \eqref{eq4.4} we know that almost all the zeros of $f,$ $g$ and $h$ in the complex plane are simple zeros such that
 \begin{equation}\label{eq4.32}
 N_{(2}\left(r,\frac{1}{f}\right)+ N_{(2}\left(r,\frac{1}{g}\right)+ N_{(2}\left(r,\frac{1}{h}\right)=S(r).
 \end{equation}
By \eqref{eq4.32} we have
\begin{equation}\label{eq4.33}
 N\left(r,\frac{1}{f'}\right)=N_0\left(r,\frac{1}{f'}\right)+S(r),
 \end{equation}
\begin{equation}\label{eq4.34}
 N\left(r,\frac{1}{g'}\right)=N_0\left(r,\frac{1}{g'}\right)+S(r),
 \end{equation}
and
\begin{equation}\label{eq4.35}
 N\left(r,\frac{1}{h'}\right)=N_0\left(r,\frac{1}{h'}\right)+S(r).
 \end{equation}
\vskip 2mm
\par Let $z_0\in\Bbb{C}$ be a common zero of $f'$ and $g'$ that is neither a zero of $f$ nor a zero of $g,$ and that $z_0$ is a simple zero of $h$ at least. By \eqref{eq4.8} we have \eqref{eq4.22} and \eqref{eq4.23}. This is discussed as follows:
\vskip 2mm
\par Suppose that $h(z_0)\neq 0.$ Then, by \eqref{eq4.22} we deduce
\begin{equation}\label{eq4.36}
\min\left\{U\left(\frac{1}{f'},z_0\right),  U\left(\frac{1}{g'},z_0\right)\right\}\leq U\left(\frac{1}{\alpha},z_0\right).
 \end{equation}
\vskip 2mm
\par Suppose that $z_0$ is a simple zero of $h.$ Then, by the supposition we deduce that the multiplicity of $z_0$ as a zero of $(h^8)' (f^8)''-(h^8)'' (f^8)'$ is equal to $7$ at least. By \eqref{eq4.22} and \eqref{eq4.23} we have
$$(f^8)'(g^8)''-(f^8)''(g^8)'=(h^8)' (f^8)''-(h^8)'' (f^8)'.$$
Therefore, the multiplicity of $z_0$ as a zero of $(f^8)'(g^8)''-(f^8)''(g^8)'$ is also equal to $7$ at least. This together with \eqref{eq4.22} and the supposition that $z_0$ is a simple zero of $h$ implies that $z_0$ is a zero of $\alpha.$ Therefore, by \eqref{eq4.22} we deduce
\begin{equation}\label{eq4.37}
\min\left\{U\left(\frac{1}{f'},z_0\right),  U\left(\frac{1}{g'},z_0\right)\right\}\leq 7U\left(\frac{1}{\alpha},z_0\right).
 \end{equation}
By \eqref{eq4.36} and \eqref{eq4.37} we have
\begin{equation}\label{eq4.38}
\min\left\{U\left(\frac{1}{f'},z_0\right),  U\left(\frac{1}{g'},z_0\right)\right\}\leq U\left(\frac{1}{\alpha},z_0\right)+U\left(\frac{1}{h^6},z_0\right)\leq 7U\left(\frac{1}{\alpha},z_0\right).
 \end{equation}
\vskip 2mm
\par By Claim I(i) we know that $\frac{f^7f'}{g^7g'}+1$ and $-h^7h'$ share $0$ CM*. Combining this with \eqref{eq4.3}-\eqref{eq4.5}, \eqref{eq4.22}, \eqref{eq4.32}-\eqref{eq4.38} and the second fundamental theorem,  we have
\begin{equation}\label{eq4.39}
\begin{split}
    &\quad 7T(r,g)+N\left(r,\frac{1}{g'}\right)+S(r)\\
    &=7\overline{N}_{1)}\left(r,\frac{1}{g}\right)+N\left(r,\frac{1}{g'}\right)
     -N_{(2}\left(r,\frac{1}{h}\right)-7N\left(r,\frac{1}{\alpha}\right)+S(r)\\
    &\leq N\left(r,\frac{f^7f'}{g^7g'}\right)+S(r)\leq T\left(r,\frac{f^7f'}{g^7g'}\right)+S(r)\\     
    &\leq \overline{N}\left(r,\frac{f^7f'}{g^7g'}\right)+\overline{N}\left(r,\frac{g^7g'}{f^7f'}\right)+\overline{N}\left(r, \frac{1}{\frac{f^7f'}{g^7g'}+1}\right)+S(r)\\
 &=\overline{N}\left(r,\frac{f^7f'}{g^7g'}\right)+\overline{N}\left(r,\frac{g^7g'}{f^7f'}\right)+\overline{N}\left(r, \frac{1}{-h^7h'}\right)+S(r)\\
&\leq \overline{N}\left(r,\frac{1}{g}\right)+\overline{N}_0\left(r,\frac{1}{g'}\right)+\overline{N}\left(r,\frac{1}{f}\right)
+\overline{N}_0\left(r,\frac{1}{f'}\right)+\overline{N}\left(r,\frac{1}{h}\right)+\overline{N}_0\left(r,\frac{1}{h'}\right)\\
&\quad+S(r)\\
&\leq T(r,g)+T(r,f)+T(r,h)+N_0\left(r,\frac{1}{g'}\right)+N_0\left(r,\frac{1}{f'}\right)+N_0\left(r,\frac{1}{h'}\right)+S(r)\\
&\leq T(r,g)+T(r,f)+T(r,h)+N_0\left(r,\frac{1}{g'}\right)+N\left(r,\frac{f}{f'}\right)+N\left(r,\frac{h}{h'}\right)+S(r)\\
&\leq 3T(r,g)+N_0\left(r,\frac{1}{g'}\right)+T\left(r,\frac{f'}{f}\right)+T\left(r,\frac{h'}{h}\right)+S(r)\\
&\leq 3T(r,g)+N_0\left(r,\frac{1}{g'}\right)+m\left(r,\frac{f'}{f}\right)+N\left(r,\frac{f'}{f}\right)+m\left(r,\frac{h'}{h}\right)+N\left(r,\frac{h'}{h}\right)\\
&\quad+S(r)\\
&= 3T(r,g)+N_0\left(r,\frac{1}{g'}\right)+N\left(r,\frac{f'}{f}\right)+N\left(r,\frac{h'}{h}\right)+S(r)\\
&=3T(r,g)+N_0\left(r,\frac{1}{g'}\right)+\overline{N}\left(r,f\right)+\overline{N}\left(r,\frac{1}{f}\right)+\overline{N}\left(r,h\right)+\overline{N}\left(r,\frac{1}{h}\right)+S(r)\\
&\leq 3T(r,g)+N_0\left(r,\frac{1}{g'}\right)+2T(r,f)+2T(r,h)+S(r)\\
 &\leq 7T(r,g)+N_0\left(r,\frac{1}{g'}\right)+S(r)\leq 7T(r,g)+N\left(r,\frac{1}{g'}\right)+S(r).
 \end{split}
  \end{equation}
By \eqref{eq4.39} we obtain that all the inequalities in \eqref{eq4.39} are equalities that are equal to $7T(r,g)+N\left(r,\frac{1}{g'}\right)+S(r).$ This gives
\begin{equation}\label{eq4.40}
N\left(r,\frac{1}{g'}\right)=N_0\left(r,\frac{1}{g'}\right)+S(r)= \overline{N}_0\left(r,\frac{1}{g'}\right)+S(r),
\end{equation}
\begin{equation}\label{eq4.41}
2T(r,f)=N_0\left(r,\frac{1}{f'}\right)+S(r)= \overline{N}_0\left(r,\frac{1}{f'}\right)+S(r)
\end{equation}
and
\begin{equation}\label{eq4.42}
2T(r,h)=N_0\left(r,\frac{1}{h'}\right)+S(r)= \overline{N}_0\left(r,\frac{1}{h'}\right)+S(r).
\end{equation}
By Claim I(ii) we know that $\frac{g^7g'}{h^7h'}+1$ and $-f^7f'$ share $0$ CM*.  Next, in the same manner as the above, we can deduce by \eqref{eq4.3}-\eqref{eq4.5}, \eqref{eq4.15}, \eqref{eq4.32}-\eqref{eq4.35} and the second fundamental theorem that
\begin{equation}\label{eq4.43}
\begin{split}
    &\quad 7T(r,h)+N\left(r,\frac{1}{h'}\right)+S(r)\\
   &=7\overline{N}_{1)}\left(r,\frac{1}{h}\right)+N\left(r,\frac{1}{h'}\right)-N_{(2}\left(r,\frac{1}{f}\right)-7N\left(r,\frac{1}{\alpha}\right)+S(r)\\
&\leq N\left(r,\frac{g^7g'}{h^7h'}\right)+S(r)\leq T\left(r,\frac{g^7g'}{h^7h'}\right)+S(r)\\
&\leq \overline{N}\left(r,\frac{g^7g'}{h^7h'}\right)+\overline{N}\left(r,\frac{h^7h'}{g^7g'}\right)+\overline{N}\left(r, \frac{1}{\frac{g^7g'}{h^7h'}+1}\right)+S(r)\\
&=\overline{N}\left(r,\frac{g^7g'}{h^7h'}\right)+\overline{N}\left(r,\frac{h^7h'}{g^7g'}\right)+\overline{N}\left(r, \frac{1}{-f^7f'}\right)+S(r)\\
&\leq \overline{N}\left(r,\frac{1}{h}\right)+\overline{N}_0\left(r,\frac{1}{h'}\right)+\overline{N}\left(r,\frac{1}{g}\right)
+\overline{N}_0\left(r,\frac{1}{g'}\right)+\overline{N}\left(r,\frac{1}{f}\right)+\overline{N}_0\left(r,\frac{1}{f'}\right)\\
&\quad+S(r)\\
&\leq T(r,h)+T(r,g)+T(r,f)+N_0\left(r,\frac{1}{h'}\right)+N_0\left(r,\frac{1}{g'}\right)+N_0\left(r,\frac{1}{f'}\right)+S(r)\\
&\leq T(r,h)+T(r,g)+T(r,f)+N_0\left(r,\frac{1}{h'}\right)+N\left(r,\frac{g}{g'}\right)+N\left(r,\frac{f}{f'}\right)+S(r)\\
&\leq 3T(r,h)+N_0\left(r,\frac{1}{h'}\right)+T\left(r,\frac{g'}{g}\right)+T\left(r,\frac{f'}{f}\right)+S(r)\\
&\leq 3T(r,h)+N_0\left(r,\frac{1}{h'}\right)+m\left(r,\frac{g'}{g}\right)+N\left(r,\frac{g'}{g}\right)+m\left(r,\frac{f'}{f}\right)+N\left(r,\frac{f'}{f}\right)\\
&\quad+S(r)\\
&= 3T(r,h)+N_0\left(r,\frac{1}{h'}\right)+N\left(r,\frac{g'}{g}\right)+N\left(r,\frac{f'}{f}\right)+S(r)\\
\end{split}
\end{equation}
\begin{equation}\nonumber
\begin{split}
&=3T(r,h)+N_0\left(r,\frac{1}{h'}\right)+\overline{N}\left(r,g\right)+\overline{N}\left(r,\frac{1}{g}\right)+\overline{N}\left(r,f\right)+\overline{N}\left(r,\frac{1}{f}\right)+S(r)\\
&\leq 3T(r,h)+N_0\left(r,\frac{1}{h'}\right)+2T(r,g)+2T(r,f)+S(r)\\
&\leq 7T(r,h)+N_0\left(r,\frac{1}{h'}\right)+S(r)\leq 7T(r,h)+N\left(r,\frac{1}{h'}\right)+S(r).
 \end{split}
  \end{equation}
By \eqref{eq4.43} we obtain that all the inequalities in \eqref{eq4.43} are equalities that are equal to $7T(r,h)+N\left(r,\frac{1}{h'}\right)+S(r).$ This gives
\begin{equation}\label{eq4.44}
N\left(r,\frac{1}{h'}\right)=N_0\left(r,\frac{1}{h'}\right)+S(r)= \overline{N}_0\left(r,\frac{1}{h'}\right)+S(r)
\end{equation}
and
\begin{equation}\label{eq4.45}
2T(r,g)=N_0\left(r,\frac{1}{g'}\right)+S(r)= \overline{N}_0\left(r,\frac{1}{g'}\right)+S(r).
\end{equation}
Obviously
\begin{equation}\label{eq4.46}
\begin{split}
N_0\left(r,\frac{1}{f'}\right)&\leq N\left(r,\frac{1}{f'}\right)\leq T(r,f')+O(1)\leq 2T(r,f)+S(r),
 \end{split}
\end{equation}
\begin{equation}\label{eq4.47}
N_0\left(r,\frac{1}{g'}\right)\leq N\left(r,\frac{1}{g'}\right)\leq 2T(r,g)+S(r),
\end{equation}
and
\begin{equation}\label{eq4.48}
N_0\left(r,\frac{1}{h'}\right)\leq N\left(r,\frac{1}{h'}\right)\leq 2T(r,h)+S(r).
\end{equation}
By \eqref{eq4.40}-\eqref{eq4.42} and \eqref{eq4.44}-\eqref{eq4.48} we deduce Claim II(v)-Claim II(vii). This together with \eqref{eq4.3}, \eqref{eq4.4},  \eqref{eq4.39} and \eqref{eq4.43} gives Claim II(i)-Claim II(iv). This proves  Claim II.
\vskip 2mm
\par Noting that $f,$ $g$ and $h$ are transcendental meromorphic functions in the complex plane, we use Lemma \ref{lemma2.17} to deduce that $f$ has infinitely many repelling periodic points of minimal period $n_0$ for a given positive integer $n_0$ such that $n_0\geq 2.$ Next we let $z_{n_0}\in \Bbb{C}\setminus\{0\}$ be a  repelling periodic point of minimal period $n_0.$ Then, it follows from Lemma \ref{lemma2.18} that  $z_{n_0}\in J(f)$ with
 \begin{equation}\label{eq4.49}
f^{\circ mn_0}(z_{n_0})=f_1^{\circ m}(z_{n_0}) =z_{n_0}
\end{equation}
and
 \begin{equation}\label{eq4.50}
 \begin{split}
\left| (f^{\circ mn_0})' (z_{n_0})\right|&=\left| \prod\limits_{j=0}^{mn_0-1}f'\left(f^{\circ j}\right) (z_{n_0})\right|=\left| (f_1^{\circ m})' (z_{n_0})\right| =\left| \prod\limits_{l=0 }^{m-1}f'_1\left(f_1^{\circ l}\right) (z_{n_0})\right|\\
&=\left| \prod\limits_{l=0 }^{m-1}f'_1\left(f_1^{\circ l}\right) (z_{n_0})\right|=\left|f'_1\left(z_{n_0}\right)\right|^{m} >1
  \end{split}
 \end{equation}
for each positive integer $m\in\Bbb{Z}^{+}.$ Here and in what follows,
\begin{equation}\label{eq4.51}
   f_m(z)=f^{\circ mn_0}(z) \, \,  \text{for} \, \, z\in D_{\delta_m}(z_{n_0})  \, \, \text{and} \, \, m\in\Bbb{Z}^{+},
\end{equation}
and  $D_{\delta_m}(z_{n_0})$ is defined as
 \begin{equation}\label{eq4.52}
 D_{\delta_m}(z_{n_0})=\left\{z\in\Bbb{C}:|z-z_{n_0}|<\delta_m\right\}
 \end{equation}
 for each positive integer $m\in\Bbb{Z}^{+}.$ Here $\delta_m$ is some small positive number satisfying $\delta_{m+1}\leq \delta_m$ for each positive integer $m\in\Bbb{Z}^{+},$ such that $f_m(z)$ is analytic in $D_{\delta_m}(z_{n_0})$ for each positive integer $m\in\Bbb{Z}^{+},$  and such that
  \begin{equation}\label{eq4.53}
 0< \frac{|z_{n_0}|}{2} \leq \left| f_m(z)\right|\leq \frac{3|z_{n_0}|}{2}, \, \, \left| f'_m(z)\right|=  \left|(f^{\circ mn_0})'(z)\right|>1
 \end{equation}
and
 \begin{equation}\label{eq4.54}
 \left|f'(f^{\circ mn_0})(z)\right|>0
\end{equation}
for each  $z \in  D_{\delta_m}(z_{n_0})$ and each positive integer $m\in\Bbb{Z}^{+}.$ This is derived from \eqref{eq4.49}-\eqref{eq4.51} and the continuity of $f(z),$ $f'(z),$ $f^{\circ mn_0}(z)$ and $(f^{\circ mn_0})' (z)$ at $z_{n_0}.$ From \eqref{eq4.52} we have
 \begin{equation}\label{eq4.55}
  D_{\delta_m}(z_{n_0})=\left\{z\in\Bbb{C}:\left|\frac{z}{\delta_m}-\frac{z_{n_0}}{\delta_m}\right|<1\right\}
  \end{equation}
for each positive integer $m\in\Bbb{Z}^{+}.$ Next we let
\begin{equation}\label{eq4.56}
\tilde{z}=\frac{z}{\delta_m}-\frac{z_{n_0}}{\delta_m}
 \end{equation}
for each positive integer $m\in\Bbb{Z}^{+}.$ Then, it follows from \eqref{eq4.55} and \eqref{eq4.56} that
\begin{equation}\label{eq4.57}
z=\delta_m\tilde{z}+z_{n_0} \, \, \text{with} \, \, |\tilde{z}|<1
 \end{equation}
for each positive integer $m\in\Bbb{Z}^{+}.$ From \eqref{eq4.49}-\eqref{eq4.53} and \eqref{eq4.57} we have
\begin{equation}\label{eq4.58}
\tilde{f}_m(\tilde{z})=: f_m(\delta_m\tilde{z}+z_{n_0})=  f^{\circ mn_0}\left(\delta_m\tilde{z}+z_{n_0}\right) \, \,  \text{with} \, \, |\tilde{z}|<1,
\end{equation}
 for each positive integer $m\in\Bbb{Z}^{+},$ and such that for each positive integer $m\in\Bbb{Z}^{+},$ we have
 \begin{equation}\label{eq4.59}
  0<\frac{|z_{n_0}|}{2} \leq \left|\tilde{f}_m(\tilde{z})\right|\leq \frac{3|z_{n_0}|}{2} \, \, \text{and} \, \,
   \left| \tilde{f}'_m(\tilde{z})\right|=  \delta_m\left|(f^{\circ mn_0})'(z)\right|>\delta_m>0
 \end{equation}
with $\delta_m\tilde{z}+z_{n_0}\in  D_{\delta_m}(z_{n_0})$ and  $|\tilde{z}|<1.$ Noting that $z_{n_0}\in J(f),$ we deduce by \eqref{eq4.51}, \eqref{eq4.52} and the definition of the Julia set $J=J(f)$ of the transcendental meromorphic function $f$ that the family $\{ f_m(z)\}$ with $z\in D_{\delta_m}(z_{n_0})$ for each positive integer $m\in\Bbb{Z}^{+}$ is not normal at the point $z_{n_0}.$ Accordingly, it follows from \eqref{eq4.57} and \eqref{eq4.58} that the family $\left\{\tilde{f}_m(\tilde{z})\right\}$ defined in $|\tilde{z}|<1$ is not normal at the origin point $\tilde{z}=0.$ Combining this with Lemma \ref{lemma2.16}, we see that there exist (i) points $\tilde{z}_m$ in $|\tilde{z}|<1$ such that $\tilde{z}_m\rightarrow 0,$ (ii) positive numbers $\tilde{\rho}_m$ such that $\tilde{\rho}_m\rightarrow 0^{+}$
and (iii) an infinite subsequence of $\left\{\tilde{f}_m(\tilde{z})\right\},$ say $\left\{\tilde{f}_m(\tilde{z})\right\}$ itself
such that
 \begin{equation}\label{eq4.60}
 \tilde{f}_m(\tilde{z}_m+\tilde{\rho}_m\zeta)=f_m(\delta_m(\tilde{z}_m+\tilde{\rho}_m\zeta)+z_{n_0})=  f^{\circ mn_0}\left(\delta_m(\tilde{z}_m+\tilde{\rho}_m\zeta)+z_{n_0}\right) \rightarrow
\tilde{g}(\zeta)
\end{equation}
spherically uniformly on compact subset of $\Bbb{C}.$ Here and in what follows,
\begin{equation}\label{eq4.61}
\delta_m(\tilde{z}_m+\tilde{\rho}_m\zeta)+z_{n_0}\in  D_{\delta_m}(z_{n_0}) \, \, \text{and} \, \, |\tilde{z}_m+\tilde{\rho}_m\zeta|<1
\end{equation}
for each $\zeta\in\Bbb{C}$ and the large positive integer $m\in \Bbb{Z}^{+},$ and $\tilde{g}$ is a non-constant meromorphic function. The function $\tilde{g}$
may be taken to satisfy the normalization $\tilde{g}^{\#}(\zeta)\leq\tilde{g}^{\#}(0)=1.$ Moreover, from \eqref{eq4.54} and \eqref{eq4.61} we have
 \begin{equation}\label{eq4.62}
\left|f'\left(f^{\circ mn_0}\left(\delta_m(\tilde{z}_m+\tilde{\rho}_m\zeta)+z_{n_0}\right)\right)\right|=\left|f'\left(f^{\circ mn_0}\left(\delta_m(\tilde{z}_m+\tilde{\rho}_m\zeta)+z_{n_0}\right)\right)\right|>0
 \end{equation}
for each $\zeta\in\Bbb{C}$ and the large positive integer $m$ such that $z=\delta_m\tilde{z}+z_{n_0} \in  D_{\delta_m}(z_{n_0})$ with $\tilde{z}=\tilde{z}_m+\tilde{\rho}_m\zeta$ and $|\tilde{z}|<1.$  Next we prove the following claim:
 \vskip 2mm
\par{\bf Claim III}. Under the assumptions of Theorem \ref{Theorem1.10}, the function $\tilde{g}(\zeta)$ defined in \eqref{eq4.60} is a transcendental entire function such that $(f(\tilde{g}))'(\zeta)\neq 0$ for each $\zeta\in \Bbb{C}.$ Here and in what follows, $(f(\tilde{g}))'$ denotes the first derivative of the composition $f(\tilde{g})$ of $f$ and $\tilde{g}.$
\vskip 2mm
\par We prove Claim III: \, first of all, we suppose that there exists a point $\zeta_0\in\Bbb{C}$ such that $\tilde{g}(\zeta_0)=\infty.$ Then, there exists some positive number $\delta_0$ such that $\tilde{g}(\zeta)$ is analytic in $\mathring{\overline{D}}_{\delta_0}(\zeta_0)=:\left\{\zeta\in\Bbb{C}:0<\left|\zeta-\zeta_0\right|\leq \delta_0\right\}$ and such that
 $\left|\tilde{g}(\zeta)\right|>0$ for each $\zeta\in\mathring{\overline{D}}_{\delta_0}(\zeta_0).$ Combining this with the known result
 $ \frac{1}{\tilde{f}_m(\tilde{z}_m+\tilde{\rho}_m\zeta)}\rightarrow\frac{1}{\tilde{g}(\zeta)}$ spherically uniformly on the closed disk  $\overline{D}_{\delta_0}(\zeta_0)=:\left\{\zeta\in\Bbb{C}:\left|\zeta-\zeta_0\right|\leq\delta_0\right\},$ where $\frac{1}{\tilde{g}(\zeta)}$
 and $\frac{1}{\tilde{f}_m(\tilde{z}_m+\tilde{\rho}_m\zeta)}$ with $m\in\Bbb{Z}^{+}$ are analytic functions on $\overline{D}_{\delta_0}(\zeta_0),$  and $\zeta_0$ is the only one zero of $\frac{1}{\tilde{g}(\zeta)}$ in $D_{\delta_0}(\zeta_0)=:\left\{\zeta\in\Bbb{C}:\left|\zeta-\zeta_0\right|<\delta_0\right\},$ by Hurwitz's theorem (cf.\cite[p.9 Hurwitz Theorem]{Schiff1993}) we deduce that $\frac{1}{\tilde{f}_m(\tilde{z}_m+\tilde{\rho}_m\zeta)}$ has one zero in $D_{\delta_0}(\zeta_0)$ at least for the large positive integer $m.$ However, from  \eqref{eq4.60},  \eqref{eq4.61}, the left inequality of \eqref{eq4.59} and the supposition $z_{n_0}\neq 0,$ we deduce that
$\frac{1}{\tilde{f}_m(\tilde{z}_m+\tilde{\rho}_m\zeta)}$ has no zeros in $D_{\delta_0}(\zeta_0).$ This is a contradiction. Therefore, we prove that $\tilde{g}$
is a non-constant entire function.
\vskip 2mm
\par Secondly, we suppose that there exists a point $\tilde{\zeta}_0\in\Bbb{C}$ such that $\tilde{g}(\tilde{\zeta}_0)=0.$ Then, there exists some positive number $\tilde{\delta}_0$ such that $\tilde{g}(\zeta)$ is analytic in $\overline{D}_{\tilde{\delta}_0}(\zeta_0)=:\left\{\zeta\in\Bbb{C}:\left|\zeta-\tilde{\zeta}_0\right|\leq \tilde{\delta}_0\right\},$ and such that
 $\left|\tilde{g}(\zeta)\right|>0$ for each
 \begin{equation}\nonumber
 \zeta\in\mathring{\overline{D}}_{\tilde{\delta}_0}(\tilde{\zeta}_0)=:\left\{\zeta\in\Bbb{C}:0<\left|\zeta-\tilde{\zeta}_0\right|\leq \tilde{\delta}_0\right\}.
  \end{equation}
 Combining this with the known result that $\tilde{f}_m(\tilde{z}_m+\tilde{\rho}_m\zeta)\rightarrow\tilde{g}(\zeta)$ spherically uniformly on the closed disk  $\overline{D}_{\tilde{\delta}_0}(\tilde{\zeta}_0)=:\left\{\zeta\in\Bbb{C}:\left|\zeta-\tilde{\zeta}_0\right|\leq\tilde{\delta}_0\right\},$ where $\tilde{g}(\zeta)$  and $\tilde{f}_m(\tilde{z}_m+\tilde{\rho}_m\zeta)$ with $m\in\Bbb{Z}^{+}$ are analytic functions on $\overline{D}_{\tilde{\delta}_0}(\tilde{\zeta}_0),$  and $\tilde{\zeta}_0$ is the only one zero of $\tilde{g}(\zeta)$ in $D_{\tilde{\delta}_0}(\tilde{\zeta}_0),$ we deduce by Hurwitz's theorem (cf.\cite[p.9 Hurwitz Theorem]{Schiff1993}) that $\tilde{f}_m(\tilde{z}_m+\tilde{\rho}_m\zeta)$ has one zero in $D_{\tilde{\delta}_0}(\tilde{\zeta}_0)$ at least for the large positive integer $m.$ However, from  \eqref{eq4.60}, \eqref{eq4.61}, the left inequality of \eqref{eq4.59} and the supposition $z_{n_0}\neq 0,$ we deduce that
$\tilde{f}_m(\tilde{z}_m+\tilde{\rho}_m\zeta)$ has no zeros in $D_{\tilde{\delta}_0}(\tilde{\zeta}_0)$ for the large positive integer m. This is a contradiction. This implies that $0$ is a Picard exceptional value of $\tilde{g}$ in the complex plane. Combining this with the fundamental theorem of algebra and the known result that  $\tilde{g}(\zeta)$ is a non-constant entire function, we deduce that $\tilde{g}$ is a transcendental entire function such that $0$ is a Picard exceptional value of  $\tilde{g}.$
\vskip 2mm
\par Thirdly, we prove that $0$ is a Picard exceptional value of $(f(\tilde{g}))'.$ Indeed, from the known result that $\tilde{g}$ is a transcendental entire function and the supposition that $f$ is a transcendental meromorphic function in the complex plane, we deduce from Lemma \ref{lemma2.19} that the composition $f(\tilde{g})$ of $f$ and $\tilde{g}$ is also a transcendental meromorphic function in the complex plane. Then,
the first derivative $(f(\tilde{g}))'$ of the composition $f(\tilde{g})$ is still a transcendental meromorphic function in the complex plane. Next we suppose that there exists a point $\hat{\zeta}_0\in \Bbb{C}$ such that $(f(\tilde{g}))'(\hat{\zeta}_0)=0.$ Then, we can find some positive number $\hat{\delta}_0$ such that $(f(\tilde{g}))'$ is analytic in $\overline{D}_{\hat{\delta}_0}(\hat{\zeta}_0)=:\left\{\zeta\in\Bbb{C}:\left|\zeta-\hat{\zeta}_0\right|\leq \hat{\delta}_0\right\},$ and such that
 $\left|(f(\tilde{g}))'(\zeta)\right|>0$ for each $\zeta\in\mathring{\overline{D}}_{\hat{\delta}_0}(\hat{\zeta}_0)=:\left\{\zeta\in\Bbb{C}:0<\left|\zeta-\hat{\zeta}_0\right|\leq \hat{\delta}_0\right\}.$ On the other hand, from \eqref{eq4.60} and \eqref{eq4.61}, we deduce that
 \begin{equation}\label{eq4.63}
 \begin{split}
&\quad  f\left(\tilde{f}_m(\tilde{z}_m+\tilde{\rho}_m\zeta)\right)=  f\left(f_m(\delta_m(\tilde{z}_m+\tilde{\rho}_m\zeta)+z_{n_0})\right)\\
&=f\left(f^{\circ mn_0}\left(\delta_m(\tilde{z}_m+\tilde{\rho}_m\zeta)+z_{n_0}\right)\right)\rightarrow f(\tilde{g})(\zeta)
    \end{split}
   \end{equation}
 spherically uniformly on the closed disk $\overline{D}_{\hat{\delta}_0}(\hat{\zeta}_0).$
 Noting that $(f(\tilde{g}))'$ is analytic on the closed disk $\overline{D}_{\hat{\delta}_0}(\hat{\zeta}_0),$ we see that the
  the composition $f(\tilde{g})$ is also analytic on the closed disk $\overline{D}_{\hat{\delta}_0}(\hat{\zeta}_0).$ Combining this with the supposition that the composition $f\left(\tilde{f}_m(\tilde{z}_m+\tilde{\rho}_m\zeta)\right)$ is analytic on the closed disk  $\overline{D}_{\hat{\delta}_0}(\hat{\zeta}_0)$
 for the large positive integer $m,$ we deduce by \eqref{eq4.63} and Weierstrass theorem (cf.\cite [p.9, Weierstrass Theorem]{Schiff1993}) that
\begin{equation}\label{eq4.64}
 \begin{split}
&\quad  \frac{df\left(\tilde{f}_m(\tilde{z}_m+\tilde{\rho}_m\zeta)\right)}{d\zeta}=  \frac{df\left(f_m(\delta_m(\tilde{z}_m+\tilde{\rho}_m\zeta)+z_{n_0})\right)}{d\zeta}\\
&=\frac{df\left(f^{\circ mn_0}\left(\delta_m(\tilde{z}_m+\tilde{\rho}_m\zeta)+z_{n_0}\right)\right)}{d\zeta}\rightarrow (f(\tilde{g}))'(\zeta)
    \end{split}
   \end{equation}
 spherically uniformly on the closed disk  $\overline{D}_{\hat{\delta}_0}(\hat{\zeta}_0).$ On the other hand, since
$(f(\tilde{g}))'$ is analytic on the closed disk $\overline{D}_{\hat{\delta}_0}(\hat{\zeta}_0)$ such that $(f(\tilde{g}))'(\zeta)\not\equiv 0$ for each $\zeta\in \overline{D}_{\hat{\delta}_0}(\hat{\zeta}_0),$  and $\hat{\zeta}_0$ is the only one zero of $(f(\tilde{g}))'$ on the closed disk  $\overline{D}_{\hat{\delta}_0}(\hat{\zeta}_0),$ we deduce by \eqref{eq4.64}, Hurwitz's theorem (cf.\cite[p.9 Hurwitz Theorem]{Schiff1993}) and the supposition  that the composition $f\left(\tilde{f}_m(\tilde{z}_m+\tilde{\rho}_m\zeta)\right)$ is analytic on the closed disk  $\overline{D}_{\hat{\delta}_0}(\hat{\zeta}_0)$
 for each positive integer $m,$ we deduce that
 \begin{equation}\label{eq4.65}
 \frac{d\left(f\left(f^{\circ mn_0}\left(\delta_m(\tilde{z}_m+\tilde{\rho}_m\zeta)+z_{n_0}\right)\right)\right)}{d\zeta}
  \end{equation}
  has one zero at least in the open disk $D_{\hat{\delta}_0}(\hat{\zeta}_0)=:\left\{\zeta\in\Bbb{C}:\left|\zeta-\hat{\zeta}_0\right|<\hat{\delta}_0\right\}.$
 However, from \eqref{eq4.54}-\eqref{eq4.57}, \eqref{eq4.60}-\eqref{eq4.62} and the right formula of \eqref{eq4.59} we deduce
\begin{equation}\label{eq4.66}
\begin{split}
&\quad \left| \frac{d\left(f\left(f^{\circ mn_0}\left(\delta_m(\tilde{z}_m+\tilde{\rho}_m\zeta)+z_{n_0}\right)\right)\right)}{d\zeta}\right|\\
& =\left|\frac{df^{\circ mn_0}\left(\delta_m(\tilde{z}_m+\tilde{\rho}_m\zeta)+z_{n_0}\right)}{d\zeta}\cdot f'\left(f^{\circ mn_0}\left(\delta_m(\tilde{z}_m+\tilde{\rho}_m\zeta)+z_{n_0}\right)\right)\right|\\
&=\delta_m\tilde{\rho}_m\left|(f^{\circ mn_0})'\left(\delta_m(\tilde{z}_m+\tilde{\rho}_m\zeta)+z_{n_0}\right)\right|\cdot\left|f'\left(f^{\circ mn_0}\left(\delta_m(\tilde{z}_m+\tilde{\rho}_m\zeta)+z_{n_0}\right)\right)\right|\\
&>\delta_m\tilde{\rho}_m\left|f'\left(f^{\circ mn_0}\left(\delta_m(\tilde{z}_m+\tilde{\rho}_m\zeta)+z_{n_0}\right)\right)\right|>0
\end{split}
   \end{equation}
for any given  $\zeta\in\Bbb{C}$ and the corresponding large positive integer $m$  such that \eqref{eq4.61} holds. From \eqref{eq4.66} we see that  \eqref{eq4.65} has no zeros on the closed disk $\overline{D}_{\hat{\delta}_0}(\hat{\zeta}_0)$ for the corresponding large positive integer $m.$ This is a contradiction. Therefore, we prove that $0$ is a Picard exceptional value of $(f(\tilde{g}))'.$ This complete the proof of Claim III.
\vskip 2mm
\par Last, we complete the proof of Theorem \ref{Theorem1.10}. Indeed, from Claim III we know that $\tilde{g}$ in \eqref{eq4.60} is a transcendental entire function. Combining this with the supposition that the transcendental meromorphic functions $f,$ $g,$ $h$ satisfy the functional equation $f^8+g^8+h^8=1,$ we deduce that
$f(\tilde{g}),$ $g(\tilde{g})$ and $h(\tilde{g})$ are still transcendental meromorphic functions that satisfy
\begin{equation}\label{eq4.67}
(f(\tilde{g})(\zeta))^8+(g(\tilde{g})(\zeta))^8+(h(\tilde{g})(\zeta))^8=1
\end{equation}
for each $\zeta\in\Bbb{C}.$ From \eqref{eq4.67} and Claim II(V)-Claim II(vii) we have
\begin{equation}\label{eq4.68}
2T\left(r,f(\tilde{g})\right)=N\left(r,\frac{1}{\left(f(\tilde{g})\right)'}\right)+S_1(r),
\end{equation}
where and in what follows,  $S_1(r)$ is any quantity such that $S_1(r)=o (T_1(r)),$ as $r\not\in E_1$ and $r\rightarrow\infty.$ Here $E_1\subset (0,+\infty)$ is a set of finite linear measure, and $T_1(r)=T(r,f(\tilde{g}))+T(r,g(\tilde{g}))+T(r,h(\tilde{g})).$ From \eqref{eq4.67} and Lemma \ref{lemma2.11} we have
\begin{equation}\label{eq4.69}
 \begin{split}
 \frac{1}{3}T_1(r)&=T(r,f(\tilde{g})))+S(r)=T(r,g(\tilde{g})))+S_1(r)=T(r,h(\tilde{g})))+S_1(r)\\
&=\overline{N}\left(r,\frac{1}{f(\tilde{g}))}\right)+S(r)=\overline{N}\left(r,\frac{1}{g(\tilde{g}))}\right)+S(r)=\overline{N}\left(r,\frac{1}{h(\tilde{g}))}\right)
+S_1(r)\\
&=\overline{N}(r,f(\tilde{g})))+S_1(r)=\overline{N}(r,g(\tilde{g})))+S_1(r)=\overline{N}(r,h(\tilde{g})))+S_1(r)\\
 &=\overline{N}_0(r,f(\tilde{g})), g(\tilde{g})),h(\tilde{g})))+S_1(r).
 \end{split}
 \end{equation}
From Claim III we have
\begin{equation}\label{eq4.70}
N\left(r,\frac{1}{\left(f(\tilde{g})\right)'}\right)=0.
\end{equation}
We substituting \eqref{eq4.70} into \eqref{eq4.68}, and then we have $T(r,f(\tilde{g}))=S_1(r),$
which contradicts \eqref{eq4.69}. Theorem \ref{Theorem1.10} is thus completely proved.
\vskip 2mm
\par
\section{\bf Proof of Theorem \ref{Theorem1.11}}
 Since the proof is similar to the proof of Theorem \ref{Theorem1.10}, we omit some details below.
\vskip 2mm
\par{\bf Case 1.} Suppose that $f^6,$ $g^6$ and $h^6$ are linearly dependent. Then, in the same manner as in Case 1 in the proof of Theorem \ref{Theorem1.10} we can get a contradiction.
\vskip 2mm
\par{\bf Case 2.} Suppose that $f^6,$ $g^6$ and $h^6$ are linearly independent. First of all, by Lemma \ref{lemma2.12} we can deduce that $f,$ $g$ and $h$ share $0$ CM* such that \eqref{eq4.4} holds. Also by Lemma \ref{lemma2.12} we have \eqref{eq4.5}. By \eqref{eq4.5} we can see that
 any two of the functions $f,$ $g$ and $h$ almost have no common zeros. Moreover, by \eqref{eq4.4} we deduce
\begin{equation}\label{eq5.1}
 \begin{split}
 \frac{1}{3}T(r)&=T(r,f)+S(r)=T(r,g)+S(r)=T(r,h)+S(r)\\
&=\overline{N}_{1)}\left(r,\frac{1}{f}\right)+S(r)=\overline{N}_{1)}\left(r,\frac{1}{g}\right)+S(r)=\overline{N}_{1)}\left(r,\frac{1}{h}\right)+S(r)
 \end{split}
 \end{equation}
 and
\begin{equation}\label{eq5.2}
 \overline{N}_{(1,1)}\left(r,\frac{1}{f}, \frac{1}{g}\right)+\overline{N}_{(1,1)}\left(r,\frac{1}{g}, \frac{1}{h}\right)
  +\overline{N}_{(1,1)}\left(r,\frac{1}{h}, \frac{1}{f}\right)=S(r).
\end{equation}
\vskip 2mm
\par On the other hand, by \eqref{eq2.1}, \eqref{eq2.7} and Theorem \ref{TheoremB} we have for $n=m=k=6$ that
\begin{equation} \label {eq5.3}
 W(f^6, g^6, h^6)=\left|\aligned (f^6)' &\quad  (g^6)'\\
                                   (f^6)'' &\quad  (g^6)''
                                   \endaligned
                                  \right|= \left|\aligned (g^6)' &\quad  (h^6)'\\
                                  (g^6)'' &\quad  (h^6)''\endaligned\right| = \left|\aligned (h^6)' &\quad  (f^6)'\\
                                  (h^6)'' &\quad  (f^6)''\endaligned\right|= \beta f^4g^4h^4,
\end{equation}
where $\beta$ is a small function of $f,$ $g$ and $h$ such that $\beta\not\equiv 0, \infty.$
\vskip 2mm
\par On the other hand, by taking the first derivatives of two sides of
\eqref{eq1.1} for $n=6$ we have
\begin{equation}\label {eq5.4}
f^5f'+g^5g'+h^5h'=0,
\end{equation}
and so we have
\begin{equation}\label{eq5.5}
g^5g'\left(\frac{f^5f}{g^5g'}+1\right)=-h^5h'
\end{equation}
and
\begin{equation}\label {eq5.6}
\frac{f^5f}{g^5g'}+\frac{h^5h'}{g^5g'}+1=0.
\end{equation}
In the same manner as in the proof of Claim I and by \eqref{eq4.4}, \eqref{eq4.5} and \eqref{eq5.1}-\eqref{eq5.6} we get the following claim:
\vskip 2mm
\par {\bf  Claim IV.} The following items hold: (i) $\frac{f^5f'}{g^5g'}+1$ and $-h^5h'$ share $0$ CM*; (ii) $\frac{g^5g'}{h^5h'}+1$ and $-f^5f'$ share $0$ CM*; (iii) $\frac{h^5h'}{f^5f'}+1$ and $-g^5g'$ share $0$ CM*.
\vskip 2mm
\par Next we prove the following claim:
\vskip 2mm
\par {\bf Claim V.} \, The following equalities hold:
\begin{equation}\nonumber
\begin{split}
\rm{(i)}\quad &T\left(r,\frac{f^5f'}{g^5g'}\right)=\overline{N}\left(r,\frac{f^5f'}{g^5g'}\right)+\overline{N}\left(r,\frac{g^5g'}{f^5f'}\right)
+\overline{N}\left(r, \frac{1}{\frac{f^5f'}{g^5g'}+1}\right)+S(r),\\
\rm{(ii)}\quad &T\left(r,\frac{g^5g'}{h^5h'}\right)=\overline{N}\left(r,\frac{g^5g'}{h^5h'}\right)+\overline{N}\left(r,\frac{h^5h'}{g^5g'}\right)
+\overline{N}\left(r,\frac{1}{\frac{g^5g'}{h^5h'}+1}\right)+S(r),\\
\rm{(iii)} \quad &T\left(r,\frac{h^5h'}{f^5f'}\right)=\overline{N}\left(r,\frac{h^5h'}{f^5f'}\right)+\overline{N}\left(r,\frac{f^5f'}{h^5h'}\right)
+\overline{N}\left(r,\frac{1}{\frac{h^5h'}{f^5f'}+1}\right)+S(r),\\
\rm{(iv)} \quad  &2T(r) =T\left(r,\frac{f^5f'}{g^5g'}\right)+S(r)=T\left(r,\frac{g^5g'}{h^5h'}\right)+S(r)=T\left(r,\frac{h^5h'}{f^5f'}\right)+S(r),\\
\rm{(v)}\quad &T(r,f)=N\left(r,\frac{1}{f'}\right)+S(r)=N_0\left(r,\frac{1}{f'}\right)+S(r)= \overline{N}_0\left(r,\frac{1}{f'}\right)+S(r),\\
\rm{(vi)}\quad  &T(r,g)=N \left(r,\frac{1}{g'}\right)+S(r)=N_0\left(r,\frac{1}{g'}\right)+S(r)= \overline{N}_0\left(r,\frac{1}{g'}\right)+S(r),\\
\rm{(vii)}\quad &T(r,h)=N\left(r,\frac{1}{h'}\right)=N_0\left(r,\frac{1}{h'}\right)+S(r)= \overline{N}_0\left(r,\frac{1}{h'}\right)+S(r).
 \end{split}
  \end{equation}
 \vskip 2mm
\par Indeed, by Claim IV(i) we know that $\frac{f^5f}{g^5g'}+1$ and $-h^5h'$ share $0$ CM*, and by \eqref{eq5.1} we know that almost all the zeros of $f,$ $g$ and $h$ in the complex plane are simple zeros such that
 \begin{equation}\label{eq5.7}
 N_{(2}\left(r,\frac{1}{f}\right)+ N_{(2}\left(r,\frac{1}{g}\right)+ N_{(2}\left(r,\frac{1}{h}\right)=S(r).
 \end{equation}
By \eqref{eq5.7} we have
\begin{equation}\label{eq5.8}
 N\left(r,\frac{1}{f'}\right)=N_0\left(r,\frac{1}{f'}\right)+S(r),
 \end{equation}
\begin{equation}\label{eq5.9}
 N\left(r,\frac{1}{g'}\right)=N_0\left(r,\frac{1}{g'}\right)+S(r),
 \end{equation}
and
\begin{equation}\label{eq5.10}
 N\left(r,\frac{1}{h'}\right)=N_0\left(r,\frac{1}{h'}\right)+S(r).
 \end{equation}
 On the other hand, by \eqref{eq5.3} we have
\begin{equation}\label {eq5.11}
\left|\aligned (f^6)' &\quad  (g^6)'\\
 (f^6)'' &\quad  (g^6)''
 \endaligned\right|=(f^6)'(g^6)''-(g^6)'(f^6)''= \beta f^4g^4h^4
 \end{equation}
\begin{equation}\label {eq5.12}
\left|\aligned (g^6)' &\quad  (h^6)'\\
 (g^6)'' &\quad  (h^6)''
 \endaligned\right|=(g^6)'(h^6)''-(h^6)'(g^6)''= \beta f^4g^4h^4
 \end{equation}
and
\begin{equation}\label {eq5.13}
 \left|\aligned (h^6)' &\quad  (f^6)'\\
 (h^6)'' &\quad  (f^6)''\endaligned\right|=(h^6)'(f^6)''-(f^6)'(h^6)''= \beta f^4g^4h^4.
  \end{equation}
\vskip 2mm
\par
Now we let $z_0\in\Bbb{C}$ be a common zero of $f'$ and $g'$ that is neither a zero of $f$ nor a zero of $g,$ and that $z_0$ is a simple zero of $h$ at most. This is discussed as follows:
\vskip 2mm
\par Suppose that $h(z_0)\neq 0.$ Then, by \eqref{eq5.11} and the supposition we deduce that $z_0$ is a zero of $\beta$ such that
\begin{equation}\label{eq5.14}
\min\left\{U\left(\frac{1}{f'},z_0\right),  U\left(\frac{1}{g'},z_0\right)\right\}\leq U\left(\frac{1}{\beta},z_0\right).
 \end{equation}
\vskip 2mm
\par Suppose that $z_0$ is a simple zero of $h.$ Then, by the supposition we deduce that the multiplicity of $z_0$ as a zero of $(h^6)' (f^6)''-(h^6)'' (f^6)'$ is equal to $5$ at least. By \eqref{eq5.11} and \eqref{eq5.13} we have
$$(f^6)'(g^6)''-(g^6)'(f^6)''=(h^6)' (f^6)''- (f^6)'(h^6)''.$$
Therefore, the multiplicity of $z_0$ as a zero of $(f^6)'(g^6)''-(g^6)'(f^6)''$ is also equal to $5$ at least. This together with \eqref{eq5.11} and the supposition that $z_0$ is a simple zero of $h$ implies that $z_0$ is a zero of $\beta.$ Therefore, by \eqref{eq5.11} we deduce
\begin{equation}\label{eq5.15}
\min\left\{U\left(\frac{1}{f'},z_0\right),  U\left(\frac{1}{g'},z_0\right)\right\}\leq 5U\left(\frac{1}{\beta},z_0\right).
 \end{equation}
By \eqref{eq5.14} and \eqref{eq5.15} we have
\begin{equation}\label{eq5.16}
\min\left\{U\left(\frac{1}{f'},z_0\right),  U\left(\frac{1}{g'},z_0\right)\right\}\leq U\left(\frac{1}{\beta},z_0\right)+U\left(\frac{1}{h^4},z_0\right)\leq 5U\left(\frac{1}{\beta},z_0\right).
 \end{equation}
\vskip 2mm
\par By Claim IV(i) we know that $\frac{f^5f'}{g^5g'}+1$ and $-h^5h'$ share $0$ CM*. Combining this with \eqref{eq4.4}, \eqref{eq4.5},  \eqref{eq5.3}, \eqref{eq5.7}-\eqref{eq5.10}, \eqref{eq5.16} and the second fundamental theorem, we have
\begin{equation}\label{eq5.17}
\begin{split}
&\quad 5T(r,g)+N\left(r,\frac{1}{g'}\right)+S(r)=5\overline{N}_{1)}\left(r,\frac{1}{g}\right)+N\left(r,\frac{1}{g'}\right)-N_{(2}\left(r,\frac{1}{h}\right)\\
&\quad -5N\left(r,\frac{1}{\beta}\right)+S(r)\leq N\left(r,\frac{f^5f'}{g^5g'}\right)+S(r)\leq T\left(r,\frac{f^5f'}{g^5g'}\right)+S(r)\\
&\leq \overline{N}\left(r,\frac{f^5f'}{g^5g'}\right)+\overline{N}\left(r,\frac{g^5g'}{f^5f'}\right)+\overline{N}\left(r, \frac{1}{\frac{f^5f'}{g^5g'}+1}\right)+S(r)\\
&=\overline{N}\left(r,\frac{f^5f'}{g^5g'}\right)+\overline{N}\left(r,\frac{g^5g'}{f^5f'}\right)+\overline{N}\left(r, \frac{1}{-h^5h'}\right)+S(r)\\
&\leq \overline{N}\left(r,\frac{1}{g}\right)+\overline{N}_0\left(r,\frac{1}{g'}\right)+\overline{N}\left(r,\frac{1}{f}\right)
+\overline{N}_0\left(r,\frac{1}{f'}\right)+\overline{N}\left(r,\frac{1}{h}\right)+\overline{N}_0\left(r,\frac{1}{h'}\right)\\
&\quad+S(r)\\
&\leq T(r,g)+T(r,f)+T(r,h)+N_0\left(r,\frac{1}{g'}\right)+N_0\left(r,\frac{1}{f'}\right)+N_0\left(r,\frac{1}{h'}\right)+S(r)\\
&\leq T(r,g)+T(r,f)+T(r,h)+N_0\left(r,\frac{1}{g'}\right)+N\left(r,\frac{f}{f'}\right)+N\left(r,\frac{h}{h'}\right)+S(r)\\
&\leq 3T(r,g)+N_0\left(r,\frac{1}{g'}\right)+T\left(r,\frac{f'}{f}\right)+T\left(r,\frac{h'}{h}\right)+S(r)\\
&\leq 3T(r,g)+N_0\left(r,\frac{1}{g'}\right)+m\left(r,\frac{f'}{f}\right)+N\left(r,\frac{f'}{f}\right)+m\left(r,\frac{h'}{h}\right)+N\left(r,\frac{h'}{h}\right)\\
&\quad+S(r)= 3T(r,g)+N_0\left(r,\frac{1}{g'}\right)+N\left(r,\frac{f'}{f}\right)+N\left(r,\frac{h'}{h}\right)+S(r)\\
&= 3T(r,g)+N_0\left(r,\frac{1}{g'}\right)+\overline{N}\left(r,\frac{1}{f}\right)+\overline{N}\left(r,\frac{1}{h}\right)+S(r)\\
&\leq 3T(r,g)+N_0\left(r,\frac{1}{g'}\right)+T(r,f)+T(r,h)+S(r)\\
&\leq 5T(r,g)+N_0\left(r,\frac{1}{g'}\right)+S(r)\leq 5T(r,g)+N\left(r,\frac{1}{g'}\right)+S(r).
 \end{split}
 \end{equation}
By \eqref{eq5.17} we obtain that all the inequalities in \eqref{eq5.17} are equalities that are equal to $5T(r,g)+N\left(r,\frac{1}{g'}\right)+S(r).$ This gives
\begin{equation}\label{eq5.18}
N\left(r,\frac{1}{g'}\right)=N_0\left(r,\frac{1}{g'}\right)+S(r)= \overline{N}_0\left(r,\frac{1}{g'}\right)+S(r),
\end{equation}
\begin{equation}\label{eq5.19}
T(r,f)=N_0\left(r,\frac{1}{f'}\right)+S(r)= \overline{N}_0\left(r,\frac{1}{f'}\right)+S(r)
\end{equation}
and
\begin{equation}\label{eq5.20}
T(r,h)=N_0\left(r,\frac{1}{h'}\right)+S(r)= \overline{N}_0\left(r,\frac{1}{h'}\right)+S(r).
\end{equation}
By Claim IV(ii) we know that $\frac{g^5g'}{h^5h'}+1$ and $-f^5f'$ share $0$ CM*.  Next, in the same manner as the above, we can deduce by \eqref{eq4.4}, \eqref{eq4.5}, \eqref{eq5.3}, \eqref{eq5.7}-\eqref{eq5.16} and the second fundamental theorem that
\begin{equation}\label{eq5.21}
\begin{split}
    &\quad 5T(r,h)+N\left(r,\frac{1}{h'}\right)+S(r)\\
 &=5\overline{N}_{1)}\left(r,\frac{1}{h}\right)+N\left(r,\frac{1}{h'}\right)-N_{(2}\left(r,\frac{1}{f}\right)-5N\left(r,\frac{1}{\beta}\right)+S(r)\\
&\leq N\left(r,\frac{g^5g'}{h^5h'}\right)+S(r)\leq T\left(r,\frac{g^5g'}{h^5h'}\right)+S(r)\\
&\leq \overline{N}\left(r,\frac{g^5g'}{h^5h'}\right)+\overline{N}\left(r,\frac{h^5h'}{g^5g'}\right)+\overline{N}\left(r, \frac{1}{\frac{g^5g'}{h^5h'}+1}\right)+S(r)\\
&=\overline{N}\left(r,\frac{g^5g'}{h^5h'}\right)+\overline{N}\left(r,\frac{h^5h'}{g^5g'}\right)+\overline{N}\left(r, \frac{1}{-f^5f'}\right)+S(r)\\
&\leq \overline{N}\left(r,\frac{1}{h}\right)+\overline{N}_0\left(r,\frac{1}{h'}\right)+\overline{N}\left(r,\frac{1}{g}\right)
+\overline{N}_0\left(r,\frac{1}{g'}\right)+\overline{N}\left(r,\frac{1}{f}\right)+\overline{N}_0\left(r,\frac{1}{f'}\right)\\
&\quad+S(r)\\
&\leq T(r,h)+T(r,g)+T(r,f)+N_0\left(r,\frac{1}{h'}\right)+N_0\left(r,\frac{1}{g'}\right)+N_0\left(r,\frac{1}{f'}\right)+S(r)\\
&\leq T(r,h)+T(r,g)+T(r,f)+N_0\left(r,\frac{1}{h'}\right)+N\left(r,\frac{g}{g'}\right)+N\left(r,\frac{f}{f'}\right)+S(r)\\
&\leq 3T(r,h)+N_0\left(r,\frac{1}{h'}\right)+T\left(r,\frac{g'}{g}\right)+T\left(r,\frac{f'}{f}\right)+S(r)\\
&\leq 3T(r,h)+N_0\left(r,\frac{1}{h'}\right)+m\left(r,\frac{g'}{g}\right)+N\left(r,\frac{g'}{g}\right)+m\left(r,\frac{f'}{f}\right)+N\left(r,\frac{f'}{f}\right)\\
&\quad+S(r)= 3T(r,h)+N_0\left(r,\frac{1}{h'}\right)+N\left(r,\frac{g'}{g}\right)+N\left(r,\frac{f'}{f}\right)+S(r)\\
\end{split}
\end{equation}
\begin{equation}\nonumber
\begin{split}
&=3T(r,h)+N_0\left(r,\frac{1}{h'}\right)+\overline{N}\left(r,\frac{1}{g}\right)+\overline{N}\left(r,\frac{1}{f}\right)+S(r)\\
&\leq 3T(r,h)+N_0\left(r,\frac{1}{h'}\right)+T(r,g)+T(r,f)+S(r)\\
&\leq 5T(r,h)+N_0\left(r,\frac{1}{h'}\right)+S(r)\leq 5T(r,h)+N\left(r,\frac{1}{h'}\right)+S(r).
 \end{split}
  \end{equation}
By \eqref{eq5.21} we obtain that all the inequalities in \eqref{eq5.21} are equalities that are equal to $5T(r,h)+N\left(r,\frac{1}{h'}\right)+S(r).$ This gives
\begin{equation}\label{eq5.22}
N\left(r,\frac{1}{h'}\right)=N_0\left(r,\frac{1}{h'}\right)+S(r)= \overline{N}_0\left(r,\frac{1}{h'}\right)+S(r)
\end{equation}
and
\begin{equation}\label{eq5.23}
T(r,g)=N_0\left(r,\frac{1}{g'}\right)+S(r)= \overline{N}_0\left(r,\frac{1}{g'}\right)+S(r).
\end{equation}
Obviously
\begin{equation}\label{eq5.24}
\begin{split}
N_0\left(r,\frac{1}{f'}\right)&\leq N\left(r,\frac{1}{f'}\right)\leq T(r,f')+O(1)\leq T(r,f)+S(r),
 \end{split}
\end{equation}
\begin{equation}\label{eq5.25}
N_0\left(r,\frac{1}{g'}\right)\leq N\left(r,\frac{1}{g'}\right)\leq T(r,g)+S(r),
\end{equation}
and
\begin{equation}\label{eq5.26}
N_0\left(r,\frac{1}{h'}\right)\leq N\left(r,\frac{1}{h'}\right)\leq T(r,h)+S(r).
\end{equation}
By \eqref{eq5.18}-\eqref{eq5.20} and \eqref{eq5.22}-\eqref{eq5.26} we deduce  Claim V(v)- Claim V(vii). This together with \eqref{eq4.4},  \eqref{eq5.17} and \eqref{eq5.21} gives  Claim V(i)- Claim V(iv).
\vskip 2mm
\par From Lemma \ref{lemma2.17} and the supposition that $f,$ $g$ and $h$ are transcendental entire functions, we deduce that $f$ has infinitely many repelling periodic points of minimal period $n_0$ for the given positive integer $n_0$ such that $n_0\geq 2.$ As in the proof of Theorem \ref{Theorem1.10},  next we let $z_{n_0}\in \Bbb{C}\setminus\{0\}$ be a repelling periodic point of minimal period $n_0,$ and we have from Lemma \ref{lemma2.18} that  $z_{n_0}\in J(f).$ As in the proof of Theorem \ref{Theorem1.10}, we have \eqref{eq4.49}-\eqref{eq4.59}, where and in what follows, $\delta_m$ is some small positive number satisfying $\delta_{m+1}\leq \delta_m$ for each positive integer $m\in\Bbb{Z}^{+}.$ Noting that $z_{n_0}\in J(f),$ we deduce by \eqref{eq4.51}, \eqref{eq4.52} and the definition of the Julia set $J=J(f)$ of the transcendental entire function $f$ that the family $\{ f_m(z)\}$ with $z\in D_{\delta_m}(z_{n_0})$ for each positive integer $m\in\Bbb{Z}^{+}$ is not normal at the point $z_{n_0}.$ Accordingly, it follows from \eqref{eq4.57} and \eqref{eq4.58} that the family $\left\{\tilde{f}_m(\tilde{z})\right\}$ defined in $|\tilde{z}|<1$ is not normal at the origin point $\tilde{z}=0.$ Combining this with Lemma \ref{lemma2.16}, we see that there exist (iv) points $\tilde{z}_m$ in $|\tilde{z}|<1,$ such that $\tilde{z}_m\rightarrow 0,$ (v) positive numbers $\tilde{\rho}_m,$ such that $\tilde{\rho}_m\rightarrow 0^{+}$
and (vi) an infinite subsequence of $\left\{\tilde{f}_m(\tilde{z})\right\},$ say $\left\{\tilde{f}_m(\tilde{z})\right\}$ itself
such that \eqref{eq4.60} holds, and such that \eqref{eq4.61} holds for each $\zeta\in\Bbb{C}$ and the large positive integer $m\in \Bbb{Z}^{+}.$ Here $\tilde{g}$ is a non-constant entire function. The function $\tilde{g}$
may be taken to satisfy the normalization $\tilde{g}^{\#}(\zeta)\leq\tilde{g}^{\#}(0)=1.$ Moreover, from \eqref{eq4.54} and \eqref{eq4.61} we have \eqref{eq4.62} for each $\zeta\in\Bbb{C}$ and the large positive integer $m$ such that $z=\delta_m\tilde{z}+z_{n_0} \in  D_{\delta_m}(z_{n_0})$ with $\tilde{z}=\tilde{z}_m+\tilde{\rho}_m\zeta$ and $|\tilde{z}|<1.$  Next we use \eqref{eq4.59}-\eqref{eq4.62} and the lines of the reasoning of the proof of Claim III in the proof of Theorem \ref{Theorem1.10} to get the following claim:
\vskip 2mm
\par{\bf Claim VI}. \,  Under the assumptions of Theorem \ref{Theorem1.11}, the function $\tilde{g}(\zeta)$ defined in \eqref{eq4.60} is a transcendental entire function such that $(f(\tilde{g}))'(\zeta)\neq 0$ for each $\zeta\in \Bbb{C}.$ Here and in what follows, $(f(\tilde{g}))'$ denotes the first derivative of the composition $f(\tilde{g})$ of the entire functions $f$ and $\tilde{g}.$
\vskip 2mm
\par Last, we complete the proof of Theorem \ref{Theorem1.11}. Indeed, from Claim VI we know that $\tilde{g}$ in \eqref{eq4.60} is a transcendental entire function. Combining this with the supposition that the transcendental entire functions $f,$ $g,$ $h$ satisfy the functional equation $f^6+g^6+h^6=1,$ we deduce that
$f(\tilde{g}),$ $g(\tilde{g})$ and $h(\tilde{g})$ are still transcendental entire functions that satisfy
\begin{equation}\label{eq5.27}
(f(\tilde{g})(\zeta))^6+(g(\tilde{g})(\zeta))^6+(h(\tilde{g})(\zeta))^6=1
\end{equation}
for each $\zeta\in\Bbb{C}.$ From \eqref{eq5.27} and Claim V(V)-Claim V(vii) we have
\begin{equation}\label{eq5.28}
T\left(r,f(\tilde{g})\right)=N\left(r,\frac{1}{(f(\tilde{g}))'}\right)+S_2(r),
\end{equation}
where and in what follows,  $S_2(r)$ is any quantity such that $S_2(r)=o (T_2(r)),$ as $r\not\in E_2$ and $r\rightarrow\infty.$ Here $E_2\subset (0,+\infty)$ is a set of finite linear measure, and $T_2(r)=T(r,f(\tilde{g}))+T(r,g(\tilde{g}))+T(r,h(\tilde{g})).$ From \eqref{eq5.27} and Lemma \ref{lemma2.12} we have
\begin{equation}\label{eq5.29}
 \begin{split}
 \frac{1}{3}T_2(r)&=T(r,f(\tilde{g})))+S(r)=T(r,g(\tilde{g})))+S_2(r)=T(r,h(\tilde{g})))+S_2(r)\\
&=\overline{N}\left(r,\frac{1}{f(\tilde{g}))}\right)+S(r)=\overline{N}\left(r,\frac{1}{g(\tilde{g}))}\right)+S(r)=\overline{N}\left(r,\frac{1}{h(\tilde{g}))}\right)\\
&\quad+S_2(r).
 \end{split}
 \end{equation}
From Claim VI we have
\begin{equation}\label{eq5.30}
N\left(r,\frac{1}{\left(f(\tilde{g})\right)'}\right)=0.
\end{equation}
We substituting \eqref{eq5.30} into \eqref{eq5.28}, and then we have $T(r,f(\tilde{g}))=S_2(r),$
which contradicts \eqref{eq5.29}. Theorem \ref{Theorem1.11} is thus completely proved.
\vskip 2mm
\par
\section{\bf Applications of the main results in this paper}
\vskip 2mm
\par In 2003, Gundersen-Tohge \cite{Gundersen2003} proved the following two results:
\vskip 2mm
\par
 \begin{theorem}\rm{}(\cite[Remarks concerning Question 1, pp.236-237]{Gundersen2003}) \label {Theorem6.1} \textit{ Suppose that $f$ and $g$ are two non-constant rational functions such that $f^n+1$ and  $g^n+1$ share $0$ CM, where $n\geq 8$ is a positive integer. Then $f^n=g^n$ or $f^ng^n=1.$}
 \end{theorem}
\vskip 2mm
\par
\begin{theorem}\rm{}(\cite[Remarks concerning Question 1, pp.236-237]{Gundersen2003}) \label {Theorem6.2} \textit{ Suppose that $f$ and $g$ are two non-constant
polynomials such that $f^n+1$ and  $g^n+1$ share $0$ CM, where $n\geq 6$ is a positive integer. Then $f^n=g^n$ or $f^ng^n=1.$}
 \end{theorem}
\vskip 2mm
\par From Corollary \ref{corollary1.12} and Corollary \ref{corollary1.13} we get the following two results that improve Theorem \ref{Theorem6.1} and  Theorem \ref{Theorem6.2} respectively:
\vskip 2mm
\par
 \begin{theorem}\rm{} \label {Theorem6.3} \textit{ Suppose that $f$ and $g$ are two non-constant meromorphic functions such that $f^n+1$ and  $g^n+1$ share $0$ CM, where $n\geq 8$ is a positive integer. Then $f^n=g^n$ or $f^ng^n=1.$}
 \end{theorem}
\vskip 2mm
\par
\begin{theorem}\rm{} \label {Theorem6.4} \textit{ Suppose that $f$ and $g$ are two non-constant
entire functions such that $f^n+1$ and  $g^n+1$ share $0$ CM, where $n\geq 6$ is a positive integer. Then $f^n=g^n$ or $f^ng^n=1.$}
 \end{theorem}
\vskip 2mm
\par From Corollary \ref{corollary1.12} and Corollary \ref{corollary1.13} we also get the following two results:
 \par
\vskip 2mm
 \begin{theorem} \label {Theorem6.5} \rm{}
\textit{ There do not exist non-constant meromorphic solutions $f,$ $g$ and $h$ satisfying the functional equation $f^8(z)+g^8(z)+h^8(z)=z$ for $z\in \Bbb{C}.$ }
 \end{theorem}
 \begin{theorem} \label {Theorem6.6}\rm{ }
\textit{There do not exist non-constant entire solutions $f,$ $g$ and $h$ satisfying the functional equation $f^6(z)+g^6(z)+h^6(z)=z$ for $z\in \Bbb{C}.$}
\end{theorem}
\vskip 2mm
\par{\bf Acknowledgements.} 
The authors also want to express their thanks to Professor Gary G Gundersen for his valuable suggestions and comments on several drafts of the paper. The first author met Professor Gundersen at the workshop on complex analysis on  December 6, 2016 at Taiyuan University of Technology.
\def\cprime{$'$}
\providecommand{\bysame}{\leavevmode\hbox to3em{\hrulefill}\thinspace}
\providecommand{\MR}{\relax\ifhmode\unskip\space\fi MR }
\providecommand{\MRhref}[2]{%
  \href{http://www.ams.org/mathscinet-getitem?mr=#1}{#2}
}
\providecommand{\href}[2]{#2}


\begin{thebibliography}{99}
\bibitem{Bergweiler1993} W. Bergweiler, \emph{\it Iteration of meromorphic Functions},   Bull. Amer. Math. Soc. \textbf{29}(1993), no.2,  151-188.
\bibitem{Cherry2001} W. Cherry and Z. Ye, \emph{\it Nevanlinna$'$s theory of value distribution. The second main theorem and its error
terms}, Springer Monographs in Mathematics, Springer-Verlag, Berlin, 2001.
\bibitem{Clunie1970} J. Clunie, \emph{\it The composition of entire and meromorphic functions}, Mathematical Essays Dedicated to A. J. Macintyre,
  Ohio Univ. Press, Athens, Ohio, 1970, pp.75-92.
\bibitem{Cremer1928}H. Cremer, \emph{\it Zum Zentrumproblem}, Math. Ann. \textbf{98}(1928), no.1, 151-163.
\bibitem{Cremer1938}H. Cremer, \emph{\it \"{U}ber die H\"{a}ufigkeit der Nichtzentren}, Math. Ann. \textbf{115}(1938), no.1, 573-580.
\bibitem{Fujimoto1974}H. Fujimoto, \emph{\it On meromorphic maps into the complex projective space}, J. Math. Soc. Japan \textbf{26}(1974), no.2, 272-288.
 \bibitem{Goldberg2008} A. A. Goldberg and I. V. Ostrovskii, \emph{\it Value distribution of meromorphic functions}, Translations of Mathematical Monographs, vol. 236, American Mathematical Society, Providence, RI, 2008. Translated from the 1970 Russian original by Mikhail Ostrovskii; With an appendix by Alexandre Eremenko and James K. Langley.
\bibitem{Green1975} M. Green, \emph{\it Some Picard theorems for holomorphic maps to algebraic varieties}, Amer. J. Math. \textbf{97}(1975), no. 1,  43-75.
\bibitem{Gross1966}F. Gross, \emph{\it On the equation $f^n+g^n=1$}, Bull. Amer. Math. Soc. \textbf{72}(1966), no. 1, 86-88.
\bibitem{Gundersen1998} G. G. Gundersen, \emph{\it Meromorphic solutions of $f^6+g^6+h^6 = 1,$ } Analysis (Munich) \textbf{18}(1998), no.3, 285-290.
\bibitem{Gundersen2001}G. G. Gundersen, \emph{\it Meromorphic solutions of $f^5+g^5+h^5 = 1,$} Complex Variables Theory Appl. \textbf{43}(2001), no.3-4, 293-298.
\bibitem{Gundersen2003}G. G. Gundersen, \emph{\it Complex functional equations, Complex differential and functional
equations (Mekrij\"{a}rvi, 2000)}, Univ. Joensuu Dept. Math. Rep. Ser. 5 (Univ. Joensuu,
Joensuu, 2003), 21-50.
\bibitem{GundersenHayman2004} G. G. Gundersen and W. K. Hayman,  \emph{\it The Strength of Cartan's Version of Nevanlinna Theory,} Bull. London Math. Soc.
\textbf{36}(2004), no. 4, 433-454.
\bibitem{Gundersen2003}G. G. Gundersen and K. Tohge, \emph{\it Unique range sets for polynomials or rational
functions}, Progress in Analysis, 2003, pp. 235-246.
\bibitem{Gundersen2004}G. G. Gundersen and K. Tohge, \emph{\it Entire and meromorphic solutions of $f^5 +g^5 +h^5 = 1$}. In: Symposium
on Complex Differential and Functional Equations. Univ. Joensuu Dept. Math. Rep. Ser., vol. 6, 2004, 57-67.
\bibitem{Gundersen2017}G. G. Gundersen, \emph{\it Research questions on meromorphic functions
and complex differential equations}, Comput. Methods Funct. Theory, \textbf{17}(2017),  no.~2, 195-209.
\bibitem{Gu1991}Y. X. Gu, \emph{\it Normal Families of Meromorphic Functions, Sichuan Education Press},
Chengdu, 1991 (in Chinese).
\bibitem{HalburdKorhonen2006} R. G. Halburd and R. J. Korhonen,  Nevanlinna theory for the difference operator, Ann. Fenn. Math. {\bf 31}(2006), no. 2, 463-478.
\bibitem{Hayman1964}W.~K. Hayman, \emph{\it Meromorphic functions}, Clarendon Press, Oxford, 1964.
\bibitem{Hayman1985}W.~K. Hayman, \emph{\it Waring$'$s Problem f\"{u}r analytische Funktionen \rm {(German)}}, Bayer. Akad. Wiss. Math.-Natur.
Kl. Sitzungsber, 1984(1985), 1-13.
\bibitem{HuLiYang2003} P. C. Hu, P. Li and C. C. Yang,  \emph{\it Unicity of meromorphic mappings}, Kluwer Academic Publishers, Dordrecht, 2003.
\bibitem{Ishizaki2002} K. Ishizaki, \emph{\it A note on the functional equation $f^n + g^n + h^n =1$ and some complex differential equations}, Comput.
Methods Funct. Theory \textbf{2} (2002), no.~1, 67-85.
\bibitem{Laine1993} I. Laine, \emph{\it Nevanlinna Theory and Complex Differential Equations}, Walter de Gruyter, Berlin/New York, 1993.
\bibitem{Lahiri2001}I. Lahiri, \emph{\it Weighted sharing of three values and uniqueness of meromorphic functions},
Kodai Math. J.\textbf {24}(2001), no.~ 3, 421-435.
\bibitem{Marco1992} R. P\'{e}rez Marco, \emph{\it Solution compl\`{e}te au probl\`{e}me de Siegel de lin\'{e}arisation d$'$une application holomorphe au voisinage d$'$un point fixe}, S\'{e}minaire Bourbaki 753, Ast\'{e}risque
206 (1992), no. 4, 273-310.
\bibitem{Marty1931} F. Marty, \emph{\it Recherches sur la r\'{e}partition des valeurs d'une fonction m\'{e}romorphe},
Ann. Fac. Sci. Toulouse Sci. Math. Sci. Phys. (3) 23 (1931), 183-261.
 \bibitem{Mokhon1971}A.Z. Mokhon$'$ko, \emph{\it On the Nevanlinna characteristics of some meromorphic
functions}, in Theory of functions, functional analysis and their applications, Izd. Khar. Un-ta, Kharkov, \textbf {14} (1971), 83--87.
\bibitem{Montel1927} P. Montel, \emph{\it Lecons sur les familles normales de fonctions analytiques et leurs applications}, Paris, 1927.
\bibitem{Nevanlinna1929}R. Nevanlinna,  \emph{\it Le Th\'{e}or\'{e}me de Picard-Borel et la theorie des fonctions m\'{e}romorphes,}
 Paris, Gauthier-Viliars, 1929.
 \bibitem{Nevanlinna1970} R. Nevanlinna, \emph{\it Analytic functions}, Translated from the second German edition by Phillip Emig. Die Grundlehren der mathematischen Wissenschaften, Band 162, Springer-Verlag, New York/Berlin, 1970.
 \bibitem{Newman1979} D. J. Newman and M. Slater, \emph{\it Waring$'$s problem for the ring of polynomials}, J. Number Theory,\textbf {11} (1979), no. ~ 4, 477-487.
\bibitem{Pang1988}X. C. Pang, \emph{\it Normality conditions for
differential polynomials}, Kexue Tongbao (in Chinese) \textbf{33}(1988), no. 22, 1690-1693.
\bibitem{Pang1989}X. C. Pang, \emph{\it Bloch's principle and normal criterion}, Sci. China Ser. A {\bf 32},  no. 7, (1989), 782-791.\bibitem{Schiff1993} J. Schiff, \emph{\it Normal Families}, Springer-Verlag, New York, 1993.
\bibitem{Siegel1942} C. L. Siegel, \emph{\it Iteration of analytic functions}, Ann. Math. \textbf{43}(1942), no.4, 607-612.
\bibitem{Toda1971} N. Toda, \emph{\it On the functional equation $\sum\limits_{i=0}^{p}a_i f^{n_i}_{i}$},  T\^{o}hoku Math. J. \textbf{23}(1971), no. ~ 2, 289-299.
\bibitem{YangYi2003} C. C. Yang and H. X. Yi, \emph{\it Uniqueness Theory of Meromorphic Functions}, Kluwer Academic Publishers, Dordrecht/Boston/London, 2003.
\bibitem{Yang1993} L. Yang,  \emph{\it Value Distribution Theory}, Springer-Verlag, Berlin Heidelberg, 1993.
\bibitem{YangZhang2008}L. Z. Yang and J. L. Zhang, \emph{\it Non-existence of meromorphic solutions of a Fermat type functional equation}, Aequ. math. 76 (2008), no. 1,  140-150.
\bibitem{Yang2010} L. Z. Yang and J. L. Zhang, \emph{\it On some results of Yi and their applications in the functional equations}, Acta. Math. Sci. \textbf{30B} (2010), no. 5, 1649-1660.
\bibitem{Yi1994} H. X. Yi, \emph{\it Uniqueness of meromorphic functions and a question of Gross },  Sci. China Ser. A \textbf{37}(1994),  no.~ 7, 802-813.
\bibitem{YiYang2011} H. X. Yi and L. Z. Yang, \emph{\it On meromorphic solutions of Fermat type functional equations}, Sci. China Math.\textbf{41}(2011), no.10, 907-932 (in Chinese).
\bibitem{Yoccoz1988}J. C. Yoccoz, \emph{\it Lin\'{e}arisation des germes de diff\'{e}omorphismes holomorphes de $(C,0)$},
C. R. Acad. Sci. Paris S\'{e}r. I Math.\textbf{306}(1988), no. 1, 55-58.
\bibitem{Yoccoz1995}J. C. Yoccoz, \emph{\it Th\'{e}or\`{e}me de Siegel, nombres de Bruno et polyn\^{o}mes quadratiques}, Petits diviseurs en dimension 1. Ast\'{e}risque(1995), no. 231, 3-88.
\bibitem{Zalcman1975}L. Zalcman, \emph{\it A heuristic principle in complex function theory}, Amer. Math. Monthly {\bf 82}(1975), no. 8, 813-817.
\end{thebibliography}
\end{document}